\documentclass[reqno]{amsart}

\usepackage{amsmath}
\usepackage{amssymb}
\usepackage{verbatim}
\usepackage{dsfont}
\usepackage{enumerate}
\usepackage[all]{xy}
\usepackage[mathscr]{eucal}

\usepackage[usenames]{color}

\usepackage{amsthm}
\usepackage{stmaryrd}
\usepackage{amscd}
\usepackage{amsfonts}
\usepackage{latexsym}

\usepackage{amscd}

\usepackage{graphicx}

\usepackage{setspace}

\usepackage{wasysym}

\usepackage[usenames,dvipsnames]{xcolor}

\usepackage{todonotes}

\usepackage{url}

\makeatletter

\newtheorem{theorem}{Theorem}[section]
\newtheorem{lemma}[theorem]{Lemma}
\newtheorem{proposition}[theorem]{Proposition}
\newtheorem{corollary}[theorem]{Corollary} 
\theoremstyle{definition}  

\newtheorem{example}[theorem]{Example}

\newtheorem{remark}[theorem]{Remark}

\@addtoreset{equation}{section}

\newcommand{\End}{\operatorname{End}}
 
\newcommand{\Hom}{\operatorname{Hom}} 
\def\uRep{\underline{\operatorname{Re}}\!\operatorname{p}} 
\newcommand{\Rep}{\operatorname{Rep}}
\newcommand{\id}{\text{id}}

\newcommand{\g}{\mathfrak{g}}
\newcommand{\Z}{\mathbb{Z}}

\newcommand{\C}{\mathbb{C}}

\newcommand{\Sr}{\mathbb{S}}

\newcommand{\im}{\operatorname{im}}

\newcommand{\rev}{\operatorname{rev}}

\newcommand{\B}{\mathcal{B}}

\renewcommand{\k}{\Bbbk}
\newcommand{\suchthat}{\mathrel{}\middle|\mathrel{}}
\newcommand{\Cp}{\operatorname{Cap}}
\newcommand{\lift}{\operatorname{lift}}
\newcommand{\down}{\vee}
\newcommand{\up}{\wedge}
\newcommand{\cross}{{\Large\text{$\times$}}}
\newcommand{\bigo}{\bigcirc}
\newcommand{\diam}{\raisebox{-0.015in}{{\Large\text{$\Diamond$}}}}
\newcommand{\Add}{\operatorname{Add}}
\newcommand{\Rem}{\operatorname{Rem}}
\newcommand{\I}{\mathcal{I}}
\newcommand{\<}{\langle}
\renewcommand{\>}{\rangle}
\newcommand{\rk}{\operatorname{rk}}
\newcommand{\df}{\operatorname{def}}

\newcommand{\Cat}{\mathcal{C}}

\newcommand{\one}{{\mathbf{1}}}
\newcommand{\NN}{{\mathcal{N}}}

\allowdisplaybreaks

\begin{document}

\title{Thick Ideals in Deligne's category  $\uRep(O_\delta)$}

\author{Jonathan Comes}
\address{J.C.: Department of Mathematics \& Physical Sciences, The College of Idaho}
\email{jonnycomes@gmail.com}

\author{Thorsten Heidersdorf} 
\address{T.H.: Department of Mathematics, The Ohio State University}
\email{heidersdorf.1@osu.edu}

%\subjclass[2010]{17B10, 18D10}
\thanks{2010 {\it Mathematics Subject Classification}: 17B10, 18D10.}

\begin{abstract} We describe indecomposable objects in Deligne's category $\uRep(O_\delta)$ and explain how to decompose their tensor products. We then classify thick ideals in $\uRep(O_\delta)$. As an application we classify the indecomposable summands of tensor powers of the standard representation of the orthosymplectic supergroup up to isomorphism.
\end{abstract}

\maketitle

%\tableofcontents

\section{Introduction} Let $\k$ denote a field of characteristic zero. For any $\delta \in \k$ Deligne has defined universal  categories $\uRep(Gl_\delta)$ and $\uRep(O_\delta)$ which interpolate the classical representation categories of the general linear and orthogonal group. By their universal properties these categories come with tensor functors to the representation categories of the general linear supergroup $Gl(m|n)$ and the orthosymplectic supergroup $OSp(m|2n)$, $m =2\ell$ or $2\ell+1$. In the $Gl(m|n)$-case Schur-Weyl duality, established by Sergeev \cite{Serg} and by Berele-Regev \cite{BR}, gives a classification for the irreducible summands in tensor powers $V^{\otimes r}$ of the natural representation $V$ of $Gl(m|n)$. This result was extended to the case of direct summands in mixed tensor powers $V^{\otimes r} \otimes (V^*)^{\otimes s}$
%$$T(r,s) = T_{V,V^*} (r,s) = V^{\otimes r} \otimes (V^*)^{\otimes s}$$ 
by Comes-Wilson \cite{CW} using Deligne's interpolating category $\uRep(Gl_\delta)$ and independently by Brundan-Stroppel \cite{BS}. In this case the indecomposable summands, the mixed tensors, are classified by $(m|n)$-cross bipartitions.  The results in \cite{CW} were then used to classify thick ideals in $\uRep(Gl_\delta)$ in \cite{Comes}.  In this paper we first classify  thick ideals in $\uRep(O_\delta)$, and then use that classification to classify indecomposable summands of tensor powers of the natural representation of $OSp(m|2n)$.

\subsection{Deligne's category} The category $\uRep(O_\delta)$, defined by Deligne [Del1-2]\nocite{Del96, Del07},  permits the simultaneous study of the tensor powers for the orthogonal groups and the orthosymplectic supergroups. This category is constructed as the Karoubi envelope of the additive envelope of the Brauer category $\B(\delta)$ in \S\ref{Brauer category}. The objects in $\B(\delta)$ are nonnegative integers, representing the potencies of a tensor power, and the morphism spaces are spanned by Brauer diagrams of the appropriate sizes.  In particular, the endomorphism algebras in $\B(\delta)$ are precisely the Brauer algebras $B_r(\delta)$ introduced in \cite{Brauer}.  Recent results in the representation theory of the Brauer algebra play a crucial role. Up to isomorphism, the simple $B_r(\delta)$-modules are parametrized by the set
 \begin{align*} \Lambda_r = \begin{cases}  \{ \lambda \ | \ |\lambda| = r - 2i, \ 0 \leq i \leq \frac{r}{2} \}, \text{ if } \delta \neq 0 \text{ or } r \text{ odd;} \\  \{ \lambda  \ | \ |\lambda| = r - 2i, \ 0 \leq i < \frac{r}{2} \}, \text{ if } \delta = 0 \text{ and } r \text{ even.} \end{cases} \end{align*}
This enables the parametrization of the indecomposable objects in $\uRep(O_\delta)$ by partitions (Theorem \ref{indecomposable classification}).  We write $R(\lambda$) for the indecomposable object corresponding to the partition $\lambda$.  Using a result of Koike \cite{Koike}, we obtain a formula for decomposing $R(\lambda)\otimes R(\mu)$ in terms of Littlewood-Richardson coefficients, which holds in $\uRep(O_\delta)$ for generic values of $\delta$ (Theorem \ref{generic tensor}).  
The decomposition numbers for the Brauer algebras were determined in \cite{Mar}, and they are described in \cite{CD} using a generalization of the cap diagrams introduced by Brundan-Stroppel [BS1-5]\nocite{BS1,BS2,BS3,BS4}.  We exploit this result in \S\ref{lifts} in order to describe a ``lifting isomorphism'' in the spirit of \cite{CO11}. The lifting isomorphism relates the additive Grothendieck rings of $\uRep(O_\delta)$ in the singular and generic cases (Theorem \ref{lift}(1)), thus enabling the decomposition of the tensor product of indecomposable objects in $\uRep(O_\delta)$ for any value of $\delta$.

\subsection{Classification of thick ideals}\label{intro:ideals} A collection of objects $\I$ in a braided monoidal category $\Cat$ is called a \emph{thick ideal} of $\Cat$ if the following conditions are satisfied:
\begin{enumerate}
\item[(i)] $X\otimes Y\in\I$ whenever $X\in\Cat$ and $Y\in\I$.
\item[(ii)] If $X\in\Cat$, $Y\in\I$ and there exist $\alpha:X\to Y$, $\beta:Y\to X$ such that $\beta\circ\alpha=\id_X$, then $X\in\I$.
\end{enumerate}
When $\delta\not\in\Z$, $\uRep(O_\delta)$ is semisimple (see \S\ref{ss}) and hence has no interesting thick ideals.  
In \S\ref{ideals} we classify the thick ideals in $\uRep(O_\delta)$ when $\delta\in\Z$.  More precisely, in \S\ref{lifts} we use cap diagrams to define a number $k(\lambda)$ for each partition $\lambda$.  The cap diagrams, and hence $k(\lambda)$, depend on the parameter $\delta\in\Z$.   For each integer $k \geq 0$ we define a subset $\I_k$ of $\uRep(O_\delta)$ as the collection of all objects whose indecomposable summands $R(\lambda)$ satisfy $k(\lambda) \geq k$.  In \S\ref{ideal subsection} we prove that $\I_k$ is a thick ideal (Corollary \ref{Ik is thick}).  In \S\ref{ideal classification subsection} we show that all nonzero thick ideals in $\uRep(O_\delta)$ are of the form $\I_k$ (Theorem \ref{I_k theorem}).

%%%%%%%

\subsection{The orthosymplectic supergroup}\label{intro: OSp}

We assume now that $\k$ is algebraically closed. By the universal property of Deligne's category there is a tensor functor $F_{m|2n} : \uRep(O_{m-2n}) \to \Rep(OSp(m|2n))$. We denote by $\I(m|2n)$ the collection of objects in $\uRep(O_{m-2n})$ sent to zero by $F_{m|2n}$. Then $\I(m|2n)$ is a thick ideal in $\uRep(O_{m-2n})$. Theorem \ref{theorem:kernel}, the main result of \S \ref{orthosymplectic}, states that every $\I_k$ is one of the $\mathcal{I}(m|2n)$, more precisely $\mathcal{I}(m|2n) = \I_{\min(\ell,n) + 1}$. The crucial tool here are the cohomological tensor functors of Duflo-Serganova \cite{Duflo-Serganova}. These are tensor functors $F_x: \Rep(OSp(m|2n)) \to \Rep(OSp(m-2r|2n-2r))$ with $1 \leq r \leq \min(\ell,n)$ where $x$ is an element in the odd part of the Lie superalgebra $\mathfrak{osp}(m|2n)$ satisfying $[x,x] = 0$. For a special choice of $x$ (see \S \ref{cohomological}) we show that the kernel is the thick ideal $Proj$ of projective representations. This enables us to control the image of an arbitrary thick ideal $\I$ of $\uRep(O_{m-2n})$ under the functor $F_{m|2n}$. As a consequence of a result of Lehrer-Zhang \cite[Corollary 5.8]{Lehrer-Zhang-1}, the functor $F_{m|2n}$ is full and so the image of $F_{m|2n}$ is the space of tensors of $OSp(m|2n)$, the direct sums of  summands of  tensor powers $V^{\otimes r}$. Along with results on the kernel of $F_{m|2n}$ we obtain a classification of the tensors of $\Rep(OSp(m|2n))$. For any $m,n \geq 0$ and any partition $\lambda$, $F_{m|2n}(R(\lambda))$ is an indecomposable object of $\Rep(OSp(m|2n))$ (Theorem \ref{full}) and is non-zero if and only if $\lambda$ satisfies $k(\lambda) \leq \min(\ell,n)$ (Corollary \ref{kernel}). Moreover, any non-zero indecomposable summand of a  tensor power $V^{\otimes r}$ in $\Rep(OSp(m|2n))$ is isomorphic to $F_{m|2n}(R(\lambda))$ for precisely one partition $\lambda \in \Lambda_r$ satisfying  $k(\lambda) \leq \min(\ell,n)$ (Corollary \ref{tensors}). We end the article with some open questions and problems in \S \ref{questions}.

%%%%%%%%%%%%%%%%%%%%%%%%

\subsection{Acknowledgements} The authors wish to thank Catharina Stroppel for providing many useful comments on an earlier version of this paper. The authors are also grateful to the referee for carefully reading the paper and providing useful suggestions. 

%%%%%%%%%%%%%%%%%%%%%%%

\section{Deligne's $\uRep(O_\delta)$}

\subsection{The Brauer category} \label{Brauer category}

Given $r,s\in\Z_{\geq0}$,  a \emph{Brauer diagram of type $r\to s$} is a diagrammatic presentation of a partitioning of $$\{i~|~1\leq i\leq r\}\cup\{j'~|~1\leq j\leq s\}$$ into disjoint pairs obtained by imagining $1,\ldots,r$ written below $1',\ldots,s'$ and drawing strands connecting paired elements.  For example, the following Brauer diagram of type $5\to3$  corresponds to the pairing $\{1',3'\}\cup\{1,2'\}\cup\{2,5\}\cup\{3,4\}$:
\[\includegraphics{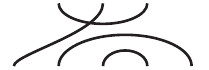}\]
We say that two Brauer diagrams are {\em equivalent}
if they are of the same type and represent the same pairing.
In diagrammatic terms, this means that one diagram can be obtained from the other by continuously deforming its strands, possibly moving them through
other strands and crossings, but keeping
endpoints fixed.

Given Brauer diagrams $g$ and $h$ of types $q\to r$ and
$r\to s$ respectively, we can stack $h$ on top of $g$ to 
obtain a Brauer diagram of type $q\to s$ along with finitely many loops made up of strands with no endpoints, which we call \emph{bubbles}.  We let $h\star g$ denote the Brauer diagram of type $q\to s$ obtained by removing all the bubbles, and let $\beta(g,h)$ denote the number of bubbles removed.  Note that $\beta(g,h)$ and $h\star g$ (up to equivalence of Brauer diagrams)  are well defined and independent of the chosen equivalence classes of $g$ and $h$.  

Now, given $\delta\in\k$ we define the \emph{Brauer category} $\B(\delta)$
to be the
category with nonnegative integers as objects and
morphisms $\Hom_{\B(\delta)}(r, s)$ consisting of all formal
$\k$-linear combinations of equivalence classes of
Brauer diagrams of type $r\to s$.
The composition of diagrams is defined by $h \circ g=\delta^{\beta(g,h)}h\star g$; it is easy to see that this is associative.
There is also a well-defined tensor product making $\B(\delta)$ into a strict 
monoidal category.
This is 
defined on diagrams so that $g \otimes h$ is obtained by horizontally stacking $g$
to the left of $h$.  The obvious braiding gives $\B(\delta)$ the structure of a strict symmetric monoidal category.  Finally, $\B(\delta)$ is rigid with every object being self-dual.  

The endomorphism algebras in $\B(\delta)$ are the well-known \emph{Brauer algebras} introduced in \cite{Brauer}.  Given $r\in\Z_{\geq0}$, we will write $B_r(\delta)$ or simply $B_r$ for the endomorphism algebra $\End_{\B(\delta)}(r)$.  

%\subsection{The functor $\B(\delta)\to\Rep(OSp(V))$}\label{Brauer functor}  
%
%\todo{State that the tensor functor exists and is full (reference [LZ]).  Mention the classical case  with $\Rep(O_d)$ as a special case.  }

%%%%%%%%%%%%%%%%%%%%%%%%%%%%%%%%%%%%%%%%%%%%%%%%%

\subsection{The category $\uRep(O_\delta)$}  

For a fixed $\delta\in\k$, the category $\uRep(O_\delta)$ is obtained from $\B(\delta)$ by formally adding direct sums and images of idempotent morphisms.  More precisely, we define the category $\uRep(O_\delta)$ to be the Karoubi envelope of the additive envelope of $\B(\delta)$  (see, for instance, \cite[\S2.5-2.6]{CW} for more details).  As explained in \emph{loc.\,cit.}, $\uRep(O_\delta)$ inherits the structure of a $\k$-linear rigid symmetric monoidal category from $\B(\delta)$, and we may identify $\B(\delta)$ with a full subcategory of $\uRep(O_\delta)$.  Moreover, whenever $e$ is an idempotent endomorphism in $\B(\delta)$, there is an object $\im e$ (the image of $e$) in $\uRep(O_\delta)$.  For idempotents $e\in\End_{\B(\delta)}(r)$ and $f\in\End_{\B(\delta)}(s)$, we will identify $\Hom_{\uRep(O_\delta)}(\im e, \im f)=f\Hom_{\B(\delta)}(r,s)e$.  

%Finally, the tensor functor $\B(\delta)\to\Rep(OSp(V))$ discussed in \S\ref{Brauer functor} extends to a full tensor functor $\uRep(O_\delta)\to\Rep(OSp(V))$.  

\section{Indecomposable objects in $\uRep(O_\delta)$}  

In this section we describe a classification of indecomposable objects in $\uRep(O_\delta)$ by partitions.  We will follow the analogous treatment of $\uRep(Gl_\delta)$ in \cite[\S4]{CW}.  In what follows we will use properties of indecomposable objects in Karoubi envelopes found in \cite[Proposition 2.7.1]{CW} without further reference.  

\subsection{Partitions, propagating strands, and symmetric groups}\label{partitions}  A \emph{partition} is an infinite tuple of integers $\lambda=(\lambda_1,\lambda_2,\ldots)$ such that $\lambda_i\geq\lambda_{i+1}\geq 0$ for each $i>0$ and $\lambda_i=0$ for all but finitely many $i$'s.  We will write $\varnothing=(0,\ldots)$.  The \emph{size} of $\lambda$ is $|\lambda|=\sum_{i>0}\lambda_i$.  The \emph{length} of $\lambda$ is the smallest integer $l(\lambda)$ such that $\lambda_{\l(\lambda)+1}=0$.  When $\lambda\not=\varnothing$, we will often write  $\lambda=(\lambda_1,\ldots,\lambda_{\l(\lambda)})$.  
Given two partitions $\lambda$ and $\mu$, we write $\lambda\subset\mu$ to mean $\lambda_i\leq \mu_i$ for all $i>0$.   We identify each partition with its Young diagram in the usual manner and write $\lambda^T$ for the transpose (or conjugate) of $\lambda$.
For example, 
$$(4,3,1)=
{\begin{picture}(55, 10)%
\put(0,-12){\line(1,0){10}}
\put(0,-2){\line(1,0){30}}
\put(0,8){\line(1,0){40}}
\put(0,18){\line(1,0){40}}
\put(0,-12){\line(0,1){30}}
\put(10,-12){\line(0,1){30}}
\put(20,-2){\line(0,1){20}}
\put(30,-2){\line(0,1){20}}
\put(40,8){\line(0,1){10}}
%\put(64,-13){\makebox(0,0){.}}
\end{picture}}
\vspace{4mm}
\text{and}
\quad
(4,3,1)^T=(3,2,2,1).
$$

%\subsection{Propagating strands and symmetric groups}    

A \emph{propagating strand} in a Brauer diagram is a strand with one top endpoint and one bottom endpoint.  For example, the only propagating stand in the Brauer diagram pictured in \S\ref{Brauer category} is the one corresponding to $\{1,2'\}$.  Given $r\in\Z_{\geq0}$, we let $\Sigma_r$ denote the set of Brauer diagrams of type $r\to r$ with $r$ propagating strands (i.e.~no cups or caps), and let $\k\Sigma_r$ denote the span of these diagrams.  It is easy to see that $\k\Sigma_r$ is a subalgebra of $B_r$ that is isomorphic to the group algebra of the symmetric group on $r$-elements.  Now, let $J\subseteq B_r$ denote the span of all Brauer diagrams with at least one non-propagating strand so that $B_r=\k\Sigma_r\oplus J$ as a vector space.  It is easy to show that $J$ is a two-sided ideal of $B_r$, hence we have an algebra map $\pi:B_r\to\k\Sigma_r$  satisfying $\pi(\sigma)=\sigma$ for each $\sigma\in\Sigma_r$.  

%%%%%%%%%%%%%%%%%%%%%%%%%%%%%%%%%%%%%%%%%%%%%%%%%

\subsection{Idempotents in $\B(\delta)$}  

It is well known that primitive idempotents in $\k\Sigma_r$ (up to conjugation) are parametrized by partitions of size $r$.  For each partition $\lambda$ of size $r$, let $z_\lambda$ denote the corresponding primitive idempotent in $\k\Sigma_r$.  Note that $z_\lambda$ is only defined up to conjugation.    We should not expect $z_\lambda$ to be primitive when viewed as an element of $B_r$.  Let $z_\lambda=e_1+\cdots+e_k$ be a decomposition of $z_\lambda$ into mutually orthogonal primitive idempotents in $B_r$.  Then $\pi(e_1),\ldots,\pi(e_k)$ are mutually orthogonal idempotents in $\k\Sigma_r$ whose sum is $z_\lambda$.  Since $z_\lambda$ is primitive in $\k\Sigma_r$, there exists a unique $i$ with $\pi(e_i)\not=0$.  Set $e_\lambda=e_i$.  
Of course, this definition of $e_\lambda$ depends on the chosen decomposition of $z_\lambda$ into mutually orthogonal idempotents. However, the conjugacy classes of the idempotents appearing in any such decomposition of $z_\lambda$ are unique. It follows that $e_\lambda$ is a primitive idempotent in $B_r$ which is uniquely defined up to conjugacy. 

\begin{example}\label{e_lambda example}
The partition $\lambda=(2)$ corresponds to the following primitive idempotent in $\k\Sigma_2$:
\[\includegraphics{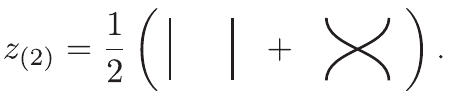}\]
One can check that $z_{(2)}$ is primitive in $B_2(\delta)$ if and only if $\delta=0$.  Hence, $e_{(2)}=z_{(2)}$ in $B_2(0)$. On the other hand, if $\delta\not=0$ we have the following idempotent in $B_2(\delta)$:
\[\includegraphics{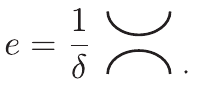}\]
If we set $e_1=e$ and $e_2=z_{(2)}-e$ then one can show $z_{(2)}=e_1+e_2$ is a decomposition of $z_{(2)}$ into mutually orthogonal idempotents. It follows that 
\[\includegraphics{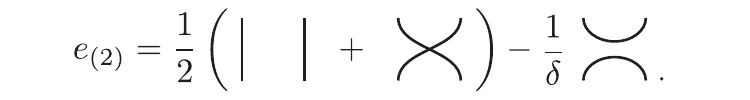}\]
\end{example}

Next we explain how to construct idempotents in $B_r$ corresponding to partitions $\lambda$ with $|\lambda|<r$.  %with size smaller than $r$.  
Write $\cup$ (resp.~$\cap$) for the unique Brauer diagram of type $0\to2$ (resp.~$2\to 0$).  Now, define the following morphisms in $\B(\delta)$ for each $i,r\in\Z_{\geq0}$:
\begin{equation*}
\psi_{r,i}=\id_r\otimes\cup^{\otimes i},
%\end{equation*}
\quad
%\begin{equation*}
\phi_{r,i}=\begin{cases}\id_{r-1}\otimes\cap^{\otimes i}\otimes\id_1, & \text{if }r>0;\\
\frac{1}{\delta^i}\,\cap^{\otimes i}, & \text{if }r=0\text{ and }\delta\not=0.
\end{cases}
\end{equation*}
Set $e_\lambda^{(i)}=\psi_{|\lambda|, i}e_\lambda\phi_{|\lambda|,i}$. In other words,
\[\includegraphics{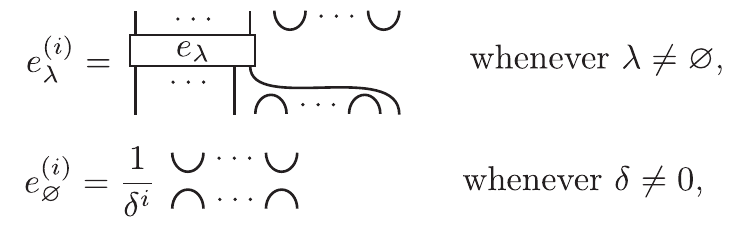}\]
where $i$ cups and $i$ caps are pictured in both diagrams above. 
Notice $e_\lambda^{(i)}\in B_{|\lambda|+2i}$ is defined for all partitions $\lambda$ and all integers $i\geq0$ except for the cases when $\lambda=\varnothing$ and $\delta=0$.  Moreover, since $\phi_{|\lambda|,i}\psi_{|\lambda|,i}=\id_{|\lambda|}$, it follows that $e_\lambda^{(i)}$ is an idempotent whenever it is defined.  

\begin{example}
If $\delta\not=0$, then it follows from Example \ref{e_lambda example} that 
\[\includegraphics{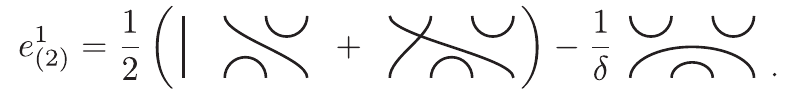}\]
\end{example}

\begin{proposition}\label{e i isomorphic} Given a partition $\lambda$, the objects $\im e_\lambda$ and $\im e_\lambda^{(i)}$ are isomorphic in $\uRep(O_\delta)$ whenever $e_\lambda^{(i)}$ is defined.
\end{proposition}

\begin{proof} Since $\phi_{|\lambda|,i}\psi_{|\lambda|,i}=\id_{|\lambda|}$, we see $\psi_{|\lambda|,i}e_\lambda=e_\lambda^{(i)}\psi_{|\lambda|,i}e_\lambda:\im e_\lambda\to \im e_\lambda^{(i)}$
%\in\Hom(\im e_\lambda, \im e_\lambda^{(i)})$ 
and $e_\lambda\phi_{|\lambda|,i}=e_\lambda\phi_{|\lambda|,i}e_\lambda^{(i)}:\im e^{(i)}_\lambda\to\im e_\lambda$
%\in\Hom(\im e^{(i)}_\lambda,\im e_\lambda)$ 
are mutually inverse isomorphisms.  
\end{proof}

Since $e_\lambda$ is primitive, the previous proposition implies that $e_\lambda^{(i)}$ is primitive whenever it is defined.  The following theorem states that, up to conjugation, these are all the primitive idempotents in Brauer algebras.   To prove the theorem, use the well-known correspondence for finite dimensional algebras between their primitive idempotents  and simple modules, along with the classification of simple modules in Brauer algebras found, for instance, in \cite[\S2]{CDM}.

\begin{theorem}\label{idempotent classification}
(1) If $\delta\not=0$ or $r$ is odd, then $\{e_\lambda^{(i)}~|~|\lambda|=r-2i, 0\leq i\leq\frac{r}{2}\}$ is a complete set of pairwise non-conjugate primitive idempotents in $B_r(\delta)$.  

(2) If $r>0$ is even, then $\{e_\lambda^{(i)}~|~|\lambda|=r-2i, 0\leq i<\frac{r}{2}\}$ is a complete set of pairwise non-conjugate primitive idempotents in $B_r(0)$.  

\end{theorem}

%%%%%%%%%%%%%%%%%%%%%%%%%%%%%%%%%%%%%%%%%%%%%%%%%

\subsection{Indecomposable objects}  

Given a partition $\lambda$, we set $R(\lambda)=\im e_\lambda$.  Since $e_\lambda$ is a primitive idempotent in $\B(\delta)$ (defined up to conjugation), $R(\lambda)$ is an indecomposable object in $\uRep(O_\delta)$ (defined up to isomorphism).  

\begin{theorem}\label{indecomposable classification}  The assignment $\lambda\mapsto R(\lambda)$ gives a bijection between the set of all partitions and the set of all isomorphism classes of nonzero indecomposable objects in $\uRep(O_\delta)$. 
\end{theorem}

\begin{proof} Every indecomposable object in $\uRep(O_\delta)$ is the image of a primitive idempotent in a Brauer algebra, hence it follows from Theorem \ref{idempotent classification} that the assignment  $\lambda\mapsto R(\lambda)$ is surjective.  To show  injectivity, assume $R(\lambda)\cong R(\mu)$ for some partitions $\lambda$ and $\mu$.  We may assume $|\lambda|\leq |\mu|$.  Since $\Hom(R(\lambda), R(\mu))=e_\mu\Hom(|\lambda|,|\mu|)e_\lambda$, there must be a Brauer diagrams of type $|\lambda|\to|\mu|$, which implies that $|\lambda|$ and $|\mu|$ have the same parity.  Set $i=\frac{|\mu|-|\lambda|}{2}$.  If $\lambda=\varnothing$, then 
$R(\lambda)\cong R(\mu)$ implies that the composition map $\Hom(0,|\mu|)\times\Hom(|\mu|,0)\to B_0(\delta)$ is nonzero.  If $\delta=0$, then this is only possible if $|\mu|=0$, so that $\lambda=\mu$.  Now we may assume $\lambda\not=\varnothing$ or $\delta\not=0$, whence $e_\lambda^{(i)}$ exists.  By Proposition \ref{e i isomorphic} we have $\im e_\lambda^{(i)}\cong\im e_\lambda\cong\im e_\mu$.  Hence $e_\lambda^{(i)}$ and $e_\mu=e_\mu^{(0)}$ are conjugate in $B_{|\mu|}(\delta)$.  Thus $\lambda$=$\mu$ by Theorem \ref{idempotent classification}.
\end{proof}

\subsection{Semisimplicity of $\uRep(O_\delta)$}\label{ss}  In \cite[Theorem 4.8.1]{CW} the known criteria for semisimplicity of the walled Brauer algebras was used to show $\uRep(Gl_\delta)$ is semisimple if and only if $\delta\not\in\Z$.  The semisimplicity of Brauer algebras is worked out in \cite{Wenzl} in the case $\k=\C$ and in \cite{Rui} for arbitrary $\k$.  Adjusting the proof of \cite[Theorem 4.8.1]{CW} accordingly, we have the following theorem (compare with \cite[Th\'eor\`eme 9.7]{Del07}).

\begin{theorem}\label{semisimple}
$\uRep(O_\delta)$ is semisimple if and only if $\delta\not\in\Z$.  
\end{theorem}

\section{Cap diagrams and the lifting map}\label{lifts}

In this section we will describe a lifting map for $\uRep(O_\delta)$ in the same manner as \cite[\S6]{CW} and \cite[\S3.2]{CO11}.  This lifting map will be used in \S\ref{tensor decomposition} to prove a formula for decomposing the tensor product $R(\lambda)\otimes R(\mu)$ in $\uRep(O_\delta)$ for generic values of $\delta$ (i.e.~for the given partitions $\lambda$ and $\mu$, the formula holds for all but finitely many $\delta\in\k$).  Moreover, the lifting map is the key ingredient for decomposing $R(\lambda)\otimes R(\mu)$ when $\delta$ is not generic.  The lifting map can be computed concretely using the cap diagrams described in \S\ref{cap diagrams}.  First, we define the lifting map in the language of additive Grothendieck rings.

\subsection{Additive Grothendieck rings}  For the remainder of this paper, we let $t$ denote an indeterminate.  Write $K_\delta$ (resp.~$K_t$) for the additive Grothendieck rings of $\uRep(O_\delta)$ over $\k$ (resp.~$\uRep(O_t)$ over $\k(t)$, the field of fractions in $t$).  More explicitly, we take elements of $K_\delta$ (resp.~$K_t$) to be formal $\Z$-linear combinations of partitions  with multiplication defined by setting $\lambda\mu=\nu^{(1)}+\cdots+\nu^{(k)}$ whenever $R(\lambda)\otimes R(\mu)=R(\nu^{(1)})\oplus\cdots\oplus R(\nu^{(k)})$ in $\uRep(O_\delta)$ (resp.~$\uRep(O_t)$).   Finally, we let $(\cdot,\cdot)_\delta$ (resp.~$(\cdot,\cdot)_t$) denote the bilinear form on $K_\delta$ (resp.~$K_t$) defined on partitions by setting $(\lambda,\mu)_\delta$ (resp.~$(\lambda,\mu)_t$) equal to $\dim_\k\Hom_{\uRep(O_\delta)}(R(\lambda),R(\mu))$ (resp.~$\dim_{\k(t)}\Hom_{\uRep(O_t)}(R(\lambda),R(\mu))$).  The following observation will be useful later.  For a proof, see \cite[Proposition 6.1.2(2)]{CW}.

\begin{equation}\label{generic form}
(\lambda,\mu)_t=\begin{cases}
1, & \text{if }\lambda=\mu;\\
0, & \text{if }\lambda\not=\mu.
\end{cases}
\end{equation}

\subsection{Lifting} 
Fix a partition $\lambda$ and consider the idempotent $e_\lambda\in B_{|\lambda|}(\delta)$. By \cite[Theorem A.3]{CO11} there exists an idempotent $\tilde{e}_\lambda=\sum_g a_g g$ with $a_g\in\k[[t-\delta]]$ for each Brauer diagram $g$ of type $|\lambda|\to|\lambda|$ such that $\tilde{e}_\lambda|_{t=\delta}=e_\lambda$. We view $\tilde{e}_\lambda$ as an idempotent in the Brauer algebra $B_{|\lambda|}(t)$ with coefficients in the field $\k((t-\delta))$.  Denote by $d_{\lambda,\mu}$ the number of idempotents corresponding to the partition $\mu$ in any decomposition of $\tilde{e}_\lambda$ into a sum of mutually orthogonal primitive idempotents in $B_{|\lambda|}(t)$. Now, set 
\begin{equation*}\label{lift def} 
    \lift_\delta(\lambda)=\sum_{\mu}d_{\lambda,\mu}\mu.
\end{equation*}
The lifting map satisfies the following properties:

\begin{theorem}\label{lift}
(1) The assignment $\lambda\mapsto \lift_\delta(\lambda)$ prescribes a $\k$-algebra isomorphism $K_\delta\to K_t$  which respects bilinear forms for all $\delta\in\k$.  

(2) For any $\lambda$, $d_{\lambda,\lambda}=1$.  Moreover, $d_{\lambda,\mu}=0$ unless $\mu=\lambda$ or $|\mu|<|\lambda|$.  

(3) For any $\lambda$, $\lift_\delta(\lambda)=\lambda$ for all but finitely many $\delta\in\k$.

(4) $(\lambda,\mu)_\delta=\sum_{\nu}d_{\lambda,\nu}d_{\mu,\nu}$ for all partitions $\lambda$ and $\mu$.  

\end{theorem}

\begin{proof} 
The fact that $\lift_\delta$ is well-defined independent from the choice of representatives for $e_\lambda$ and $\tilde{e}_\lambda$ follows from \cite[Theorem A.2]{CO11}. To show we are allowed $K_t$ as the target space as opposed to the additive Grothendieck ring of $\uRep(O_t)$ over $\k((t-\delta))$, we refer the reader to the proof of \cite[Proposition 6.1.2]{CW}. 
Now, for parts (2) and (3) we refer the reader to the proofs of the analogous statements for $\uRep(S_t)$ found in \cite[Theorem 3.12(3)-(4)]{CO11}.  To show that $\lift_\delta$ is an algebra homomorphism which respects bilinear forms, see the proofs of \cite[Theorem 3.12(1)\&(5)]{CO11}.  The fact that $\lift_\delta$ is an isomorphism follows from part (2), which completes the proof of part (1).  Part (4) is a consequence of part (1) and equation (\ref{generic form}) (see \cite[Corollary 6.2.4]{CW}).
\end{proof}

\subsection{Cap diagrams}\label{cap diagrams}  The goal of this subsection is to describe how to compute the numbers $d_{\lambda,\mu}$.  Since $\uRep(O_\delta)$ is semisimple whenever $\delta\not\in\Z$ (Theorem \ref{semisimple}) it follows from Theorem \ref{lift}(4) that $d_{\lambda,\mu}=\begin{cases}
1, & \text{if }\lambda=\mu;\\
0, & \text{if }\lambda\not=\mu.
\end{cases}$ whenever $\delta\not\in\Z$.  When $\delta\in\Z$, we will compute $d_{\lambda,\mu}$ using the cap diagrams found in \cite{Lejczyk-Stroppel} and \cite{Ehrig-Stroppel}.  To describe these diagrams, first for a partition $\lambda$ we set 
$$X_\lambda=\left\{\lambda^T_i-i+1-\frac{\delta}{2}\suchthat i>0\right\}\subseteq\frac{\delta}{2}+\Z.$$
Next, the \emph{weight diagram} for $\lambda$ is obtained by starting with a line with vertices labeled by the nonnegative elements of $\left(\frac{\delta}{2}+\Z\right)$, and marking the vertex labeled by $i$ with a 
$$\left\{\begin{array}{cl}
\diam,	& \text{ if }i=0\in X_\lambda;\\
\bigo,	& \text{ if }i\not\in X_\lambda	\text{ and }-i\not\in X_\lambda;\\
\up,		& \text{ if }i\in X_\lambda		\text{ and }-i\not\in X_\lambda;\\
\down, 	& \text{ if }i\not\in X_\lambda	\text{ and }-i\in X_\lambda;\\
\cross,	& \text{ if }i\in X_\lambda		\text{ and }-i\in X_\lambda\text{ and }i\not=0.
\end{array}\right.$$

\begin{example}\label{wt}
Here are the weight diagrams for $(4,3,3,2,1)$ with $\delta=2$ and $(6,6,6,6,4,2)$ with $\delta=3$ respectively:
$$\includegraphics{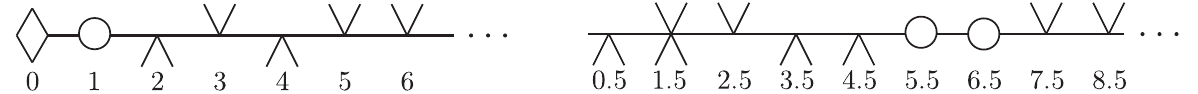}$$
\end{example}

\begin{remark}  It turns out that the number of $\bigo$'s minus the number of $\cross$'s in a weight diagram is always $\left\lfloor\delta/2\right\rfloor$.  Indeed, it is easy to verify this for the partition $\varnothing$, and adding a box to any partition does not change the difference between the number of $\bigo$'s and $\cross$'s (see Proposition \ref{shifts}).  With this in mind, one can always recover $\delta$ and the partition from its weight diagram.  
\end{remark}

The \emph{cap diagram} of $\lambda$ is constructed from its weight diagram in two stages:\begin{enumerate}
\item[]\begin{enumerate}
\item[Stage 1:] Connect vertex $j$ to vertex $i$ with a cap if $i<j$, $i$ is marked with a $\down$ or a $\diam$, $j$ is marked with a $\up$, and all vertices between $i$ and $j$ are either marked with a $\cross$ a $\bigo$ or already connected to another vertex by a cap.  Continue connecting such pairs of vertices with caps until there are no such pairs remaining.  

\item[Stage 2:] Connect vertex $i$ to vertex $j$ with a dotted cap if both $i$ and $j$ are marked with an $\up$, and all other vertices to the left of $i$ and $j$ are either marked with a $\cross$ a $\bigo$ or already connected to another vertex by a cap (dotted or un-dotted).  Continue connecting such pairs of vertices with dotted caps until there are no such pairs remaining.

\end{enumerate}
\end{enumerate}

\begin{example}\label{curl}
Here are the cap diagrams (minus the vertex labels) associated to the weight diagrams pictured in  Example \ref{wt}:
$$\includegraphics{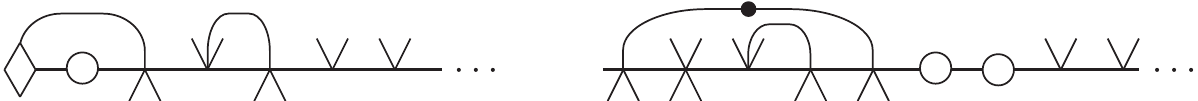}$$
\end{example}

For a given partition $\lambda$, we let $\Cp_\lambda$ denote the set of all pairs $(i,j)$ with $i<j$ such that there is a cap (dotted or un-dotted) connecting vertices $i$ and $j$ in the cap diagram of $\lambda$.  
Given $A\subseteq\Cp_\lambda$, we let $\rev_{A}(\lambda)$ denote the partition whose weight diagram is obtained by reversing the orientation on each cap (dotted or un-dotted) in $A$.  More precisely, the weight diagram of $\rev_{A}(\lambda)$ is obtained from the weight diagram of $\lambda$ by changing the marks at vertices $i$ and $j$ whenever $(i,j)\in A$ from $\up, \down, \diam$ to $\down, \up, \diam$ respectively.  Similarly, for any $\lambda$ whose weight diagram has a $\up,\down$, or $\diam$ at the $i$th vertex, we write $\rev_i(\lambda)$ for the partition whose diagram is obtained from $\lambda$ by replacing the $i$th vertex with $\down, \up$, or $\diam$ respectively.   
  Finally, given partitions $\lambda$ and $\mu$ we set
$$d'_{\lambda,\mu}=
\begin{cases} 
1, & \text{if }\mu=\rev_A(\lambda)\text{ for some }A\subseteq\Cp_\lambda;\\
0, & \text{otherwise.}
\end{cases}$$

The Brauer algebra $B_r(\delta)$ is quasi-hereditary when $\delta\not=0$ and cellular when $\delta=0$ \cite{GL}. The decomposition numbers for $B_r(\delta)$ were originally determined in \cite{Mar}. In \cite{CD} these decomposition numbers are described in terms of so-called curl diagrams. We can recover their curl diagrams from our cap diagrams by replacing all dotted caps with curls. With this in mind, it follows from \cite[Theorem 5.8]{CD} that $d'_{\lambda,\mu}$ are the decomposition numbers for Brauer algebras when $\delta\in\Z$. This will allow us to deduce $d'_{\lambda,\mu}=d_{\lambda,\mu}$ (see Corollary \ref{dees}). We will also need the following:

\begin{proposition}\label{rev size} Suppose $\lambda$ and $\mu$ are partitions and suppose the weight diagram of $\lambda$ has a $\down$ or $\up$ at the $i$th vertex.

(1) $|\rev_{i}(\lambda)|
=\begin{cases}
|\lambda|+2i, & \text{if $\lambda$ has a $\down$ at vertex $i$};\\
|\lambda|-2i, & \text{if $\lambda$ has a $\up$ at vertex $i$}.
\end{cases}$

(2) $d'_{\lambda,\lambda}=1$.  Moreover, $d'_{\lambda,\mu}=0$ unless $\mu=\lambda$ or $|\mu|=|\lambda|-2i$ for some $i\in\Z_{>0}$.  

\end{proposition}

\begin{proof} Part (2) follows from part (1).  For part (1), Set $x_l=\lambda^T_l-l+1-\frac{\delta}{2}$ for each $l>0$ so that $X_\lambda=\{x_1,x_2,\ldots\}$.  Since there is a $\down$ or $\up$ at the $i$th vertex, there exists an integer $M>0$ with 
$x_M=\begin{cases}
-i, &\text{if $\lambda$ has a $\down$ at vertex $i$};\\
i, & \text{if $\lambda$ has a $\up$ at vertex $i$}. 
\end{cases}$ 
Now, $X_{\rev_{i}(\lambda)}$ is obtained from $X_\lambda$ by replacing $x_M$  with $-x_M$.  The desired formula follows by examining the resulting effect on $|\lambda|=\sum_{l>0}(x_l+l-1+\frac{\delta}{2}$). 
\end{proof}

\begin{theorem}\label{d'd'} $(\lambda,\mu)_\delta=\sum_{\nu}d'_{\lambda,\nu}d'_{\mu,\nu}$ for all partitions $\lambda$ and $\mu$ whenever $\delta\in\Z$.
\end{theorem}

\begin{proof}
It follows from Theorem \ref{lift}(4) that the statement of the theorem is symmetric in $\lambda$ and $\mu$, hence we may assume $|\mu|\leq|\lambda|$. Since $\Hom_{\B(\delta)}(|\lambda|,|\mu|)=0$ unless $|\lambda|-|\mu|$ is divisible by 2, it follows that $(\lambda,\mu)_\delta=0$ unless $|\lambda|=|\mu|+2i$ for some $i\in\Z_{\geq 0}$. Thus, by Proposition \ref{rev size}(2) it suffices to consider the case when $|\lambda|=|\mu|+2i$. 

Suppose $\delta\not=0$. Then $B_{|\lambda|}$ is a quasi-hereditary (hence cellular) algebra (see \cite{GL}) with decomposition numbers given by $d'_{\lambda,\mu}$ (see \cite[Theorem 5.8]{CD}).  In particular, by \cite[Theorem 3.7(iii)]{GL} the projective $B_{|\lambda|}$-modules $e_\lambda B_{|\lambda|}$ and $e_\mu^{(i)}B_{|\lambda|}$ satisfy the following:
\[\dim_\k\Hom_{B_{|\lambda|}}(e_\lambda B_{|\lambda|},e_\mu^{(i)}B_{|\lambda|})=\sum_{\nu}d'_{\lambda,\nu}d'_{\mu,\nu}.\]
On the other hand,
\[\Hom_{B_{|\lambda|}}(e_\lambda B_{|\lambda|},e_\mu^{(i)}B_{|\lambda|})=e_\mu^{(i)}B_{|\lambda|}e_\lambda=\Hom_{\uRep(O_\delta)}(R(\lambda),R(\mu)).\] 
It follows that $(\lambda,\mu)_\delta=\sum_{\nu}d'_{\lambda,\nu}d'_{\mu,\nu}$ in this case.

The case when $\delta=0$ is similar, but requires a bit more care. One can prove these cases with a straightforward modification of the proof of \cite[Theorem 6.4.1]{CW}.
\end{proof}

\begin{corollary}\label{dees} (Compare with \cite[Corollary 6.4.2]{CW}) $d'_{\lambda,\mu}=d_{\lambda,\mu}$ for all partitions $\lambda$ and $\mu$ whenever $\delta\in\Z$.  
\end{corollary}

\begin{proof}
Put a partial order on pairs of partitions by declaring $(\lambda,\mu)>(\lambda',\mu')$ if either $|\lambda|>|\lambda'|$, or $|\lambda|=|\lambda'|$ and $|\mu|>|\mu|$. We prove the corollary by inducting upon this partial order. First, $d_{\varnothing,\varnothing}=1=d_{\varnothing,\varnothing}'$. Now assume $(\lambda,\mu)\not=(\varnothing,\varnothing)$. Then 
\[\begin{array}{rll} 
d_{\lambda,\mu} & = (\lambda, \mu)_\delta-\sum\limits_{\nu\atop|\nu|<|\mu|}d_{\lambda,\nu}d_{\mu,\nu} & (\text{Theorem \ref{lift}(2)\&(4)})\\
\\
& = (\lambda, \mu)_\delta-\sum\limits_{\nu\atop|\nu|<|\mu|}d'_{\lambda,\nu}d'_{\mu,\nu} & (\text{Induction})\\
& = d'_{\lambda,\mu} & (\text{Proposition \ref{rev size}(2) and Theorem \ref{d'd'}}).
\end{array}\]
\end{proof}

\begin{example} The previous corollary along with Examples \ref{wt} and \ref{curl} imply 

\noindent $\lift_2((4,3,3,2,1))=(4,3,3,2,1)+(3,3,3,2)+(4,2,1,1,1)+(3,2,1,1)$, and 

\noindent $\lift_3((6,6,6,6,4,2))=(6,6,6,6,4,2)+(6,6,6,5,4,1)+(6,6,5,2,1)+(6,6,4,2)$.
\end{example}

\begin{proposition}\label{lift box}
$\lift_\delta(\Box)=\Box$ for all $\delta$.
\end{proposition}

\begin{proof}
This follows from Corollary \ref{dees} and Proposition \ref{rev size}(2).  
\end{proof}

%%%%%%%%%%%%%%%%%%%%%%%%%%%%%%%%%%%%%%%%%%%%%%%%%

%%%%%%%%%%%%%%%%%%%%%%%%%%%%%%%%%%%%%%%%%%%%%%%%%

\section{Tensor product decomposition}\label{tensor decomposition}  

In this section we give a formula for decomposing the tensor product of indecomposable objects in $\uRep(O_t)$.  This formula along with the lifting map described in the previous section can be used to decompose tensor products in $\uRep(O_\delta)$ for arbitrary $\delta$.  To obtain the formula we interpolate a known formula for decomposing tensor products of irreducible representations of orthogonal groups. 

\subsection{Orthogonal groups}  Fix $m,\ell\in\Z_{> 0}$ with $m=2\ell$ or $m=2\ell+1$.  By the universal property of $\uRep(O_m)$ \cite[Proposition 9.4]{Del07}, the assignment $R(\Box)\mapsto V$, where $V=\k^m$ is the natural representation of $O(m)$, defines a tensor functor functor $F_m: \uRep(O_m) \to \Rep(O(m))$.  For any partition $\lambda$, let $\Sr_{[\lambda]}V$ be the representation of $O(m)$ defined in \cite[\S19.5]{FH}.

\begin{proposition}\label{image for O} If $\lambda_1^T+\lambda_2^T\leq m$, then $F_m(R(\lambda))=\Sr_{[\lambda]}V$.  
\end{proposition}

\begin{proof}  Set $r=|\lambda|$.  As in \cite[\S19.5]{FH} we write $V^{[r]}\subset V^{\otimes r}$ for the intersection of the kernels of $F_m(g)$ as $g$ runs over all Brauer diagrams in $B_r(m)$ with at least one non-propagating strand.  As explained in \emph{loc.~cit.},  $\Sr_{[\lambda]}V=\im(F_m(z_\lambda):V^{[r]}\to V^{[r]})$.  By the definition of $e_\lambda$, the actions of $F_m(z_\lambda)$ and $F_m(e_\lambda)$ on $V^{[r]}$ coincide, whence $\Sr_{[\lambda]}V$ is a submodule of $F_m(R(\lambda))$.  Now $F_m$ is full by \cite[Th\'{e}or\`{e}me 9.6]{Del07} so that  $F_m(R(\lambda))$ is indecomposable \cite[Proposition 2.7.4]{CW} and is thus  zero or irreducible in the semisimple category $\Rep(O(m))$.  The result follows since $\Sr_{[\lambda]}V$ is irreducible  when $\lambda_1^T+\lambda_2^T\leq m$ \cite[Theorem 19.19]{FH}.  
\end{proof}

\begin{remark}\label{remark on O} It turns out that  $F_m(R(\lambda))=\Sr_{[\lambda]}V$ for all $m$ and $\lambda$.  Indeed, when $\lambda_1^T+\lambda_2^T> m$ both $\Sr_{[\lambda]}V$ and $F_m(R(\lambda))$ are zero.  For the former, see \cite[Exercise 19.20]{FH}.  The latter will be discussed in \S\ref{OSp intro}.
\end{remark}

The following is a consequence of Proposition \ref{image for O} and  \cite[Theorem 19.22]{FH}:

\begin{proposition}\label{image in SO} If $l(\lambda)<\ell$, then $F_m(R(\lambda))$ restricts to the irreducible representation of $SO(m)$ with highest weight $\lambda$.  
\end{proposition}

\subsection{Generic decomposition formulae}  We are now in position to prove a formula for decomposing tensor products of indecomposable objects in $\uRep(O_t)$.  For notational convenience we work in the additive Grothendieck rings.  

\begin{theorem}\label{generic tensor} For any partitions $\lambda$ and $\mu$, the following holds in $K_t$:
\begin{equation*}
\lambda\mu=\sum_{\alpha,\beta,\gamma,\nu}LR_{\alpha,\beta}^\lambda LR_{\alpha,\gamma}^\mu LR_{\beta,\gamma}^\nu \nu,
\end{equation*}
where $LR_{\alpha,\beta}^\gamma$'s are the Littlewood Richardson coefficients.  
\end{theorem}

\begin{proof}
 Let $b_{\lambda,\mu}^\nu\in\Z$ be such that 
\begin{equation}\label{product}\lambda\mu=\sum_{\nu}b_{\lambda,\mu}^\nu\nu\end{equation} 
in $K_t$.  By Theorem \ref{lift}(3), we can find $m\in\Z$ such that (i) $\lift_m$ fixes $\lambda,\mu$, and $\nu$ whenever $b_{\lambda,\mu}^\nu\not=0$; and (ii) $m/2$ is greater than each of $l(\lambda), l(\mu)$, and $l(\nu)$ whenever $b_{\lambda,\mu}^\nu\not=0$.
%\todo{I'm not sure if this is exactly what we want.  We need $n$ to be large enough that the image of $L(\lambda), L(\mu), L(\nu)$ for $b_{\lambda,\mu}^\nu\not=0$, which I'm calling $V_\lambda, V_\mu, V_\nu$, are all nonzero.}
Assumption (i) along with Theorem \ref{lift}(1) imply that (\ref{product}) holds in $K_m$.  Hence, in $\Rep(O(m))$ we have  
$F_m(R(\lambda)) \otimes F_m(R(\mu))=\bigoplus_\nu F_m(R(\nu))^{\oplus b_{\lambda,\mu}^\nu}$.  Next, restrict to $SO(m)$ using assumption (ii) and Proposition \ref{image in SO}.  The result now follows from \cite[Theorem 3.1]{Koike}.
\end{proof}

The following is a  special case of Theorem \ref{generic tensor}:

\begin{equation}\label{generic tensor box} 
\lambda\,\Box=\sum_{\mu\in\Add(\lambda)}\mu+\sum_{\mu\in\Rem(\lambda)}\mu\qquad\text{in } K_t,
\end{equation}
where $\Add(\lambda)$ (resp.~$\Rem(\lambda)$) is the set of all partitions whose Young diagram is obtained from $\lambda$ by adding (resp.~removing) a single box.

%The following proposition will be useful later.

\begin{corollary}\label{boxes}
Any partition $\lambda$ is a summand of $\Box^{|\lambda|}$ in $K_\delta$.
\end{corollary}

\begin{proof} 
By Proposition \ref{lift box} and (\ref{generic tensor box}), $\lambda$ is a summand of $\lift_\delta(\Box^{|\lambda|})$ with maximal size.  The result now follows from Theorem \ref{lift}(2).
\end{proof}

%%%%%%%%%%%%%%%%%%%%%%%%%%%%%%%%%%%

\subsection{Interpretation via weight diagrams}  

We will repeatedly make use of (\ref{generic tensor box}).  In doing so, it will be useful to understand the effect of adding/removing boxes in terms of weight diagrams.  The following proposition is easy to verify.

\begin{proposition}\label{shifts}
Assume $\delta\in\Z$.  Adding (resp.~removing) a single box from a partition corresponds to replacing a single mark at vertex $\frac{1}{2}$ or a pair of adjacent marks with the marks above (resp.~below) them in the following table:
\begin{table}[htbp]
$\begin{array}{|c|c|c|c|c|c|c|c|c|c|c|c|c|}\hline
&&&&&&&&&&&&\\[-10pt]
\text{\Large{$\up\atop\frac{1}{2}$}}&\bigo\up&\down\bigo&\cross\up&\down\cross&\bigo\up&\diam\bigo&\bigo\cross
&\bigo\cross&\cross\bigo&\diam\up&\down\up&\down\up\\[5pt]\hline
&&&&&&&&&&&&\\[-10pt]
\text{\Large{$\down\atop\frac{1}{2}$}}&\up\bigo&\bigo\down&\up\cross&\cross\down&\diam\bigo&\bigo\down&\diam\down
&\up\down&\up\down&\bigo\cross&\bigo\cross&\cross\bigo\\[5pt]\hline
\end{array}$
\caption{}
\label{table:shifts}
\end{table}
\end{proposition}

If  $\mu\in\Add(\lambda)\cup\Rem(\lambda)$, we will refer to the corresponding column of the table above as a ``move".  For instance, we say $\mu$ is obtained from $\lambda$ by a $\up\down\leadsto\bigo\cross$ move if the weight diagrams of $\mu$ and $\lambda$ are identical except for an adjacent pair of vertices where $\lambda$ has marks $\up\down$ and $\mu$ has $\bigo\cross$.

%%%%%%%%%%%%%%%%%%%%%%%%%%%%%%%%%%%

%%%%%%%%%%%%%%%%%%%%%%%%%%%%%%%%%%%

\section{Thick ideals}\label{ideals}

%\subsection{Ideals} 

A thick ideal (see \S\ref{intro:ideals}) is called \emph{proper} if it does not contain all objects in $\Cat$.  In this section we classify the thick ideals in $\uRep(O_\delta)$.  When $\delta\not\in\Z$, $\uRep(O_\delta)$ is semisimple (Theorem \ref{semisimple}), and there are no nonzero proper thick ideals in a semisimple category.  For the remainder of this section, we assume $\delta\in\Z$.  For notational convenience, we will identify each thick ideal $\I$ in $\uRep(O_\delta)$ with subsets of $K_\delta$ so that $\lambda\in\I\subseteq K_\delta$ means $R(\lambda)\in\I\subseteq\uRep(O_\delta)$.

%Another ideal which is relevant to this paper is the following:  
%A \emph{tensor ideal}, $\mathcal{J}$, of $\Cat$ is a family of subspaces 
%$\mathcal{J}(X,Y) \subseteq \Hom_{\Cat}(X,Y)$ for all pairs of objects $X,Y$ in $\Cat$ subject to the following two conditions:
%\begin{enumerate}
%\item[(i)] $ghk \in \mathcal{J}(X,W)$ for each $k \in \Hom_{\Cat}(X,Y)$, $h \in  \mathcal{J}(Y,Z)$, and $g \in \Hom_{\Cat}(Z,W)$.
%\item[(ii)] $g\otimes \id_{Z} \in \mathcal{J}(X\otimes Z, Y\otimes Z)$ for every object $Z$ and every $g \in \mathcal{J}(X,Y)$. 
%\end{enumerate}

%%%%%%%%%%%%%%%%%%%%%%%%%%%%%%%%%%%

\subsection{The thick ideals $\I_k$}\label{ideal subsection}
For a partition $\lambda$, we define the \emph{defect} of $\lambda$, $\df(\lambda)$, to be the number of caps (dotted or un-dotted) in the cap diagram of $\lambda$.  We define the \emph{rank} of $\lambda$, $\rk(\lambda)$, to be the number of $\bigo$'s or the number of $\cross$'s in the weight diagram of $\lambda$, whichever is smaller.  Now, set $k(\lambda)=\df(\lambda)+\rk(\lambda)$.  Finally, for $k\in\Z_{\geq0}$ we define the following subset of $K_\delta$:
\begin{equation*}
\I_k=\Z_{\geq0}\{\lambda~|~k(\lambda)\geq k\}.
\end{equation*}
The main goal of this subsection is to prove that each $\I_k$ is a thick ideal (see Corollary \ref{Ik is thick}).  Towards that end, let us fix the following notation:

\begin{equation}\label{a and b}
\lambda\,\Box=\sum_\mu a_{\lambda,\mu} \mu\in K_\delta,\quad \lift_\delta(\lambda\,\Box)=\sum_\nu b_{\lambda,\nu}\nu.
\end{equation}
In particular, we have 
\begin{equation}\label{bad}
b_{\lambda,\nu}=\sum_\mu a_{\lambda,\mu}d_{\mu,\nu}.  
\end{equation}

\begin{proposition}\label{b0 implies a0}
$b_{\lambda,\nu}\geq a_{\lambda,\nu}$ for all $\lambda,\nu$.  Hence, if $b_{\lambda,\nu}=0$ then $a_{\lambda,\nu}=0$. 
\end{proposition}

\begin{proof}
This follows from (\ref{bad}) since $d_{\nu,\nu}=1$ and $a_{\lambda,\mu},d_{\mu,\nu}\geq0$.
\end{proof}

\begin{lemma} \label{b012}
For all partitions $\lambda$ and $\nu$,   $b_{\lambda,\nu}\in\{0,1,2\}$.  Moreover, if $b_{\lambda,\nu}\not=0$, then $b_{\lambda,\nu}=2$ exactly when there is some $i$ such that the vertices in the weight diagram of $\lambda$ (resp.~$\nu$) at positions $i$ and $i+1$ are $\down\up$ or $\diam\up$ (resp.~$\bigo\cross$ or $\cross\bigo$).  
\end{lemma}

\begin{proof}  
By Theorem \ref{lift}(1) and Proposition \ref{lift box}, we have $\lift_\delta(\lambda\Box)=\sum_\mu d_{\lambda,\mu}\mu\Box$.  Since $d_{\lambda,\mu}\in\{0,1\}$ (Corollary \ref{dees}), it follows from (\ref{generic tensor box}) that $b_{\lambda,\nu}$ is the number of partitions $\mu\in\Add(\nu)\cup\Rem(\nu)$ such that $d_{\lambda,\mu}=1$.  Suppose $\mu$ and $\mu'$ are two such partitions.  First, since $d_{\lambda,\mu}=1=d_{\lambda,\mu'}$, it follows from Corollary \ref{dees} that the vertices where the weight diagrams of $\mu$ and $\mu'$ differ are all labelled by $\up$ or $\down$.  Second, since $\mu,\mu'\in\Add(\nu)\cup\Rem(\nu)$, the weight diagrams of $\mu$ and $\mu'$ can be obtained from one another using exactly two of the moves from Table \ref{table:shifts}.  These two facts are only possible if $\mu=\mu'$ or the weight diagrams of $\mu$ and $\mu'$ are identical to $\nu$ except for an adjacent pair of vertices where $\nu$ has  $\bigo\cross$ or $\cross\bigo$ and either $\mu$ or $\mu'$ has  $\down\up$ (resp.~$\diam\up$) and the other has $\up\down$ (resp.~$\diam\down$).  Thus $b_{\lambda,\nu}\leq 2$.  Finally, the case $b_{\lambda,\nu}=2$ (i.e.~$\mu\not=\mu'$) occurs exactly when there is an un-dotted cap in the cap diagram of $\lambda$ connecting the vertices where the weight diagrams of $\mu, \mu'$, and $\nu$ differ.  The result follows.
\end{proof}

\begin{lemma}\label{b to a}
Suppose $\lambda$ and $\nu$ are partitions with $b_{\lambda,\nu}\not=0$.  Then there exists a partition $\eta$ with $a_{\lambda,\eta}=b_{\lambda,\nu}$, $d_{\eta,\nu}=1$, and $k(\eta)\geq k(\lambda)$.  
\end{lemma}

\begin{proof} First, since $b_{\lambda,\nu}\not=0$, then as explained in the proof of Lemma \ref{b012},  there exists $\mu\in\Add(\nu)\cup\Rem(\nu)$ with $d_{\lambda,\mu}=1$.  Hence, by Corollary \ref{dees} we have $\mu=\rev_A(\lambda)$ for some $A\subseteq\Cp_\lambda$.   Moreover, 
the weight diagrams of  $\mu$ and $\nu$ differ in at most two vertices, as prescribed by Proposition \ref{shifts}.  We will refer to these vertices as the ``crucial vertices".  Next, we explain how to construct $\eta$ using cases that depend on $\lambda$ and $\nu$ at the crucial vertices.  

We start with the cases where the crucial vertices of $\nu$ are not $\bigo\cross$ or $\cross\bigo$.  If the crucial vertices of $\nu$ are labelled by $\up\down$ or $\down\up$ (resp.~$\diam\up$ or $\diam\down$), then the crucial vertices of $\mu$ and hence of $\lambda$ are labelled by either $\bigo\cross$ or $\cross\bigo$.  In these cases, let $\eta$ be the partition whose weight diagram has $\down\up$ (resp.~$\diam\up$) at the crucial vertices, and agrees with $\lambda$ at all non-crucial vertices.  In all other cases where the crucial vertices of $\nu$ are not $\bigo\cross$ or $\cross\bigo$, there is a unique move  from Table \ref{table:shifts} at the crucial vertices in $\lambda$.  We let $\eta$ be the partition obtained from $\lambda$ by performing that move.  

Now suppose the crucial vertices of $\nu$ are $\bigo\cross$ or $\cross\bigo$.  In these cases we must choose the crucial vertices of $\eta$ to be $\bigo\cross$ or $\cross\bigo$ respectively.    
If exactly one of the crucial vertices is connected to another vertex, say the $i$th vertex, by a cap (dotted or un-dotted) in $A$, then we let the non-crucial vertices of $\eta$ be identical to $\rev_i(\lambda)$.   Otherwise, we let the non-crucial vertices of $\eta$ be identical to $\lambda$.  

With $\eta$ described above, one can check that $b_{\lambda,\eta}\not=0$, $d_{\eta,\nu}=1$, and $k(\eta)\geq k(\lambda)$ through a case by case examination of the possible weight/cap diagrams for $\nu,\mu,\lambda$, and $\eta$.  Moreover, in each case one can use Theorem \ref{lift}(2) and Proposition \ref{rev size}(1) to show that the only partition $\eta'$ with $b_{\lambda,\eta'}\not=0$ and $d_{\eta',\eta}\not=0$ is $\eta'=\eta$.  Hence, $b_{\lambda,\eta}=a_{\lambda,\eta}$ by (\ref{bad}) and Proposition \ref{b0 implies a0}.  Finally, since $d_{\eta,\nu}=1$ the weight diagrams of $\eta$ and $\nu$ have identical $\cross$'s and $\bigo$'s.  Thus, it follows from Lemma \ref{b012} that $b_{\lambda,\nu}=b_{\lambda,\eta}=a_{\lambda,\eta}$. 
\end{proof}

\begin{theorem}
Suppose $\lambda,\mu$, and $\nu$ are partitions and $\nu$ is a summand of $\lambda\mu$ in $K_\delta$.  Then $k(\nu)\geq\max\{k(\lambda), k(\mu)\}$.   
\end{theorem}

\begin{proof} Since $\mu$ is a summand of the product of $|\mu|$ copies of $\Box$ in $K_\delta$ (see Corollary \ref{boxes}), it suffices to consider the case $\mu=\Box$.  Since $k(\Box)=0$, we are required to show $k(\nu)\geq k(\lambda)$ whenever $a_{\lambda,\nu}\not=0$.  By Proposition \ref{b0 implies a0}, we may assume $b_{\lambda,\eta}\not=0$.  Now, let $\eta$ be the partition guaranteed by Lemma \ref{b to a}.  Since $a_{\lambda,\eta}=b_{\lambda,\nu}$ and $d_{\eta,\nu}=1$, it follows from (\ref{bad}) that $d_{\mu,\nu}=0$ or $a_{\lambda,\mu}=0$ whenever $\mu\not=\eta$.  Thus  $\nu=\eta$, which implies $k(\nu)=k(\eta)\geq k(\lambda)$. 
\end{proof}

The following is an immediate consequence of the previous theorem:

\begin{corollary}\label{Ik is thick}
$\I_k$ is a thick ideal in $\uRep(O_\delta)$ for each $k\geq0$.
\end{corollary}

%%%%%%%%%%%%%%%%%%%%%%%%%%%%%%%%%%%

\subsection{Classification of thick ideals}\label{ideal classification subsection}

In this subsection we will show that the $\I_k$'s are the only nonzero thick ideals in $\uRep(O_\delta)$ (see Theorem \ref{I_k theorem}).  Our arguments will be similar to those used  in \cite{Comes}.  First, we describe a special class of partitions which depend on the fixed $\delta\in\Z$:  It follows from  Proposition  \ref{shifts} that $k(\mu)\leq k(\lambda)$ whenever $\mu\in\Rem(\lambda)$.  
Given $k\in\Z_{\geq 0}$, a partition $\lambda$ is called \emph{$k$-minimal} if $k(\lambda)=k$ and each  $\mu\in\Rem(\lambda)$  satisfies $k(\mu)< k$.  The following two propositions concerning $k$-minimal partitions will be useful later:

\begin{proposition}\label{k-min diagram} A partition $\lambda$ is $k$-minimal if and only if $k=\rk(\lambda)$ and the weight diagram of $\lambda$ is of the form
$$\includegraphics{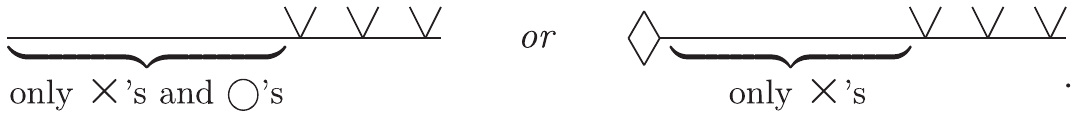}$$
  The second form above occurs when $\lambda=\varnothing$ and $\delta\in\Z_{\leq0}$ is even.
\end{proposition}

\begin{proof} Using Proposition \ref{shifts}, it is easy to check that the weight diagrams above correspond to $k$-minimal partitions.  On the other hand, if $\lambda$ is $k$-minimal, then the only moves from Table \ref{table:shifts} that remove a box from $\lambda$ and have a chance of resulting in a partition $\mu$ with $k(\mu)<k(\lambda)$ are $\bigo\cross\leadsto\diam\down$, $\bigo\cross\leadsto\up\down$, and $\cross\bigo\leadsto\up\down$.  Hence, the weight diagram of $\lambda$ cannot have  $\up$ at the vertex $\frac{1}{2}$, and cannot have any of the following adjacent pairs of vertices:  $\bigo\up, \cross\up, \diam\up, \down\up, \down\bigo, \down\cross, \diam\bigo$.  Hence, $\lambda$ cannot have its leftmost mark $\up$ and cannot have anything to the left of a $\up$, which implies $\lambda$ cannot have any $\up$'s.  Moreover, the only mark allowed to the right of a $\down$ is another $\down$, so the weight diagram of $\lambda$ must start (from the left) with a finite sequence of $\diam$'s, $\bigo$'s, and $\cross$'s, followed by an infinite sequence of $\down$'s.  In particular, $\df(\lambda)=0$, or equivalently $k(\lambda)=\rk(\lambda)$.  Finally, if a $\diam$ is present, then there cannot be any $\bigo$'s.  Otherwise, the leftmost $\bigo$ must have a $\cross$ to its left, and the move $\cross\bigo\leadsto\up\down$ will result in a partition $\mu$ whose cap diagram contains an un-dotted cap connecting the resulting $\up$ with the $\diam$ so that $\rk(\mu)=\rk(\lambda)-1$ and $\df(\mu)=1$, whence $k(\mu)=k(\lambda)=k$.  
\end{proof}

\begin{proposition}\label{k-min prop}
Suppose $\lambda$ is a partition.  If $a_{\lambda,\mu}=0$ for all $\mu\in\Rem(\lambda)$, then $\lambda$ is $k$-minimal for some $k\in\Z_{\geq 0}$.  
\end{proposition}

\begin{proof} Using Theorem \ref{lift}(1) and Proposition \ref{lift box}, followed by Corollary \ref{dees} and Proposition \ref{rev size}(2), followed by (\ref{generic tensor box}), we have 
$$\lift_\delta(\lambda\Box)=\lift_\delta(\lambda)\Box=\lambda\Box+\sum_{\eta\atop|\eta|<|\lambda|-1}d_{\lambda,\eta}\eta\Box=\sum_{\nu\in\Add(\lambda)}\nu+\sum_{\mu\in\Rem(\lambda)}\mu+\sum_{\eta\atop|\eta|<|\lambda|}c_\eta\eta$$ for some $c_\eta\in\Z$.  
Now fix $\mu\in\Rem(\lambda)$.  If $a_{\lambda,\mu}=0$ then by the calculation above along with (\ref{bad}), it follows that there exists $\nu\not=\mu$ with $d_{\nu,\mu}\not=0$ and $b_{\lambda,\nu}\not=0$.  Hence, by Theorem \ref{lift}(2), $|\nu|>|\mu|=|\lambda|-1$.  Thus, by the calculation above, $\nu\in\Add(\lambda)$.  In particular, $|\nu|=|\mu|+2$.  It now follows from Corollary \ref{dees} and Proposition \ref{rev size}(1) that $\mu=\rev_{\{(i,i+1)\}}(\nu)$ for some $i$. By examining Table \ref{table:shifts}, it is apparent that the only way for $\mu=\rev_{\{(i,i+1)\}}(\nu)$ with  $\mu\in\Rem(\lambda)$ and $\nu\in\Add(\lambda)$, is when the $i, i+1$ vertices of $\lambda$ are either $\bigo\cross$ or $\cross\bigo$.  Since $\mu$ is an arbitrary element of $\Rem(\lambda)$, it follows from Proposition \ref{shifts} that the weight diagram of $\lambda$ cannot have  $\up$ at the vertex $\frac{1}{2}$, and cannot have any of the following adjacent pairs of vertices:  $\bigo\up, \cross\up, \diam\up, \down\up, \down\bigo, \down\cross, \diam\bigo$.  Now proceed as in the proof of Proposition \ref{k-min diagram} to show that the weight diagram of $\lambda$ has the form of a $k$-minimal partition.  
\end{proof}

We will make use of the following notation throughout the rest of this section. Given a partition $\lambda$, we let $\<\lambda\>$ denote the smallest thick ideal in $K_\delta$ containing $\lambda$.  More explicitly, a partition $\mu\in\<\lambda\>$ if and only if 
$\mu$ is a summand of $\lambda\nu$ for some partition $\nu$.  Moreover, $\<\lambda\>$ is additive (i.e.~$x+y\in\<\lambda\>$ whenever $x,y\in\<\lambda\>$).  We now prove three lemmas concerning $\<\lambda\>$, which will allow us to prove the main result of this section, Theorem \ref{I_k theorem}.

\begin{lemma}\label{inclusion reverse}
 If $\mu\subset\lambda$, then 
$\<\lambda\>\subseteq\<\mu\>$.
\end{lemma}

\begin{proof}
Let $j=|\lambda|-|\mu|$.  By Theorem \ref{lift}(1)\&(2) and Proposition \ref{lift box}, $$\lift_\delta(\mu\Box^j)=\mu\Box^j+\sum_{\nu\atop|\nu|<|\mu|}d_{\mu,\nu}\nu\Box^j.$$  Hence, by (\ref{generic tensor box}), $\lambda$ is a summand of $\lift_\delta(\mu\Box^j)$ with maximal size.  Therefore, by Theorem \ref{lift}(2), $\lambda$ is a summand of $\mu\Box^j$ in $K_\delta$, hence $\lambda\in\<\mu\>$. 
\end{proof}

\begin{lemma}\label{get min}
For each partition $\lambda$, there exists $k\in\Z_{\geq0}$ and a $k$-minimal partition $\mu$ such that $\<\mu\>=\<\lambda\>$.  
\end{lemma}

\begin{proof} We induct on $|\lambda|$.  If  $\lambda$ is $k$-minimal  for some $k$, we set $\mu=\lambda$.  Otherwise, by Proposition \ref{k-min prop} there exists $\lambda'$ with $a_{\lambda,\lambda'}\not=0$ (hence $\lambda'\in\<\lambda\>$) and $\lambda'\in\Rem(\lambda)$ (hence $\lambda\in\<\lambda'\>$ by Lemma \ref{inclusion reverse}).  Since $|\lambda'|=|\lambda|-1$, by induction we can find a $k$-minimal partition $\mu$ for some $k$ with $\<\mu\>=\<\lambda'\>=\<\lambda\>$.  
\end{proof}

\begin{lemma}\label{min to min}
Fix nonnegative integers $k\leq k'$.  If $\lambda$ is $k$-minimal and $\mu$ is $k'$-minimal, then $\<\mu\>\subseteq\<\lambda\>$.  
\end{lemma}

\begin{proof}
First, set $\lambda^{(0)}=\varnothing$ and recursively define the $k$-minimal partition $\lambda^{(k)}$ for $k>0$ by modifying the weight diagram of $\lambda^{(k-1)}$ as follows:  Find the leftmost $\down$.  Since $\lambda^{(k-1)}$ is $k-1$-minimal, there are only $\cross$'s, $\bigo$'s, or a  $\diam$ to the left of that $\down$ (Proposition \ref{k-min diagram}).  Use the moves $\cross\down\leadsto\down\cross$ and $\bigo\down\leadsto\down\bigo$ until the $\down$ is at vertex $1$ or $\frac{1}{2}$.  If the $\down$ is at vertex $1$, and there is a $\diam$ at vertex $0$, then perform the  move $\diam\down\leadsto\bigo\cross$.  Otherwise, replace the $\down$ with an $\up$ and then use the moves $\up\cross\leadsto\cross\up$ and $\up\bigo\leadsto\bigo\up$ until the $\up$ has only $\down$'s to its right, and then perform the move $\up\down\leadsto\bigo\cross$.  Let $\lambda^{(k)}$ denote the partition with the resulting weight diagram.  By Proposition \ref{k-min diagram}, $\lambda^{(k)}$ is $k$-minimal.  By Proposition \ref{shifts}, $\lambda^{(k-1)}\subset\lambda^{(k)}$ for all $k>0$.  Hence, by Lemma \ref{inclusion reverse}, $\<\lambda^{(k)}\>\subseteq\<\lambda^{(k-1)}\>$ for all $k>0$.  

Now, it suffices to show $\<\mu\>=\<\lambda\>$ whenever $\lambda$ and $\mu$ are both $k$-minimal.  To do so, by Proposition \ref{k-min diagram} it suffice to show $\mu\in\<\lambda\>$ whenever the weight diagram of $\mu$ is obtained from $\lambda$ by swapping an adjacent $\cross$ and $\bigo$.  Let $\nu\in\Add(\lambda)$ be the partition whose weight diagram is obtained from $\lambda$ by a move of type $\cross\bigo\leadsto\down\up$ or $\bigo\cross\leadsto\down\up$.  Using Theorem \ref{lift}(1)\&(2) and (\ref{generic tensor box}), it follows that $\nu$ is a summand of $\lift_\delta(\lambda\Box)$ of maximal size, hence $\nu$ is a summand of $\lambda\Box$ in $K_\delta$.  Therefore $\nu\in\<\lambda\>$.  Now let $\mu\in\Rem(\nu)$ be the partition whose weight diagram is obtained from $\nu$ by a move of type $\down\up\leadsto\cross\bigo$ or $\down\up\leadsto\bigo\cross$.  In this case, $\mu$ is a summand of $\lift_\delta(\nu\Box)$ and the only summands of $\lift_\delta(\nu\Box)$ with size larger than $|\mu|$ are the partitions in $\Add(\nu)$.  It follows from Proposition \ref{shifts} and Corollary \ref{dees} that $d_{\eta,\mu}=0$ for all $\eta\in\Add(\nu)$.  Hence, by Theorem \ref{lift}(2), $\mu$ is a summand of $\nu\Box$ in $K_\delta$.  Thus $\mu\in\<\nu\>\subseteq\<\lambda\>$.  
\end{proof}

\begin{theorem}\label{I_k theorem}
For each $\delta\in\Z$, the set $\{\I_k~|~k\in\Z_{\geq0}\}$ is a complete set of pairwise distinct nonzero thick ideals in $\uRep(O_\delta)$.  
\end{theorem}

\begin{proof}  Suppose $\I$ is a nonzero thick ideal in $\uRep(O_\delta)$ and let $k=\min\{k(\lambda)~|~\lambda\in\I\}$.  We will show $\I=\I_k$.  It follows from the definition of $\I_k$ and the choice of $k$ that $\I\subseteq\I_k$.  To show the opposite inclusion, choose a nonzero $x\in\I_k$.  Then $x=\lambda^{(1)}+\cdots+\lambda^{(l)}$ for some partitions $\lambda^{(1)},\ldots,\lambda^{(l)}$ with $k(\lambda^{(i)})\geq k$ for each $i$.  By Lemma \ref{get min}, for each $1\leq i\leq l$, there exists $k_i\in\Z_{\geq0}$ and a $k_i$-minimal partition $\mu^{(i)}$ such that $\<\mu^{(i)}\>=\<\lambda^{(i)}\>\subseteq\I_k$, whence $k_i\geq k$.    Now, pick $\nu\in\I$ such that $k(\nu)=k$.  It follows from Proposition \ref{k-min prop} and the choice of $k$ that $\nu$ is $k$-minimal.  Thus, by Lemma \ref{min to min}, $\lambda^{(i)}\in\<\mu^{(i)}\>\subseteq\<\nu\>$ for all $i$.  Since $\<\nu\>$ is additive, $x\in\<\nu\>\subseteq\I$.
\end{proof}

By Andre and Kahn \cite{Andre-Kahn} every $\k$-linear rigid symmetric monoidal category has a largest proper tensor ideal $\mathcal{N}$ (the negligible morphisms). To this tensor ideal we can associate a thick ideal consisting of objects $X$ satisfying $\id_X \in \mathcal{N}(X,X)$. The associated thick ideal will again be denoted $\mathcal{N}$.  It is easy to show that $\mathcal{N}$ is the largest proper thick ideal, and that $X\in\mathcal{N}$ if and only if $\dim X=0$.

\begin{corollary}\label{k=0} We have $\mathcal{N} = \mathcal{I}_1$. Hence $\dim R(\lambda) \neq 0$ if and only if $k(\lambda) = 0$.
\end{corollary}

%%%%%%%%%%%%%%%%%%%%%%%%%%%%%%%%%%%%%%%%%%%%%%%%%

%%%%%%%%%%%%%%%%%%%%%%%%%%%%%%%%%%%%%%%%%%%%%%%%%

\section{The orthosymplectic supergroups}\label{orthosymplectic}

%Let $\k$ be algebraically closed. 

In this section we  let $st=R(\Box)$ in $\uRep(O_\delta)$ and assume $\k$ is algebraically closed.

%The analog of the standard representation in $\uRep(O_\delta)$ defined by the partition $(1,0,\ldots)$ is denoted by $st$. 

\subsection{$\uRep(O_\delta)$ and orthosymplectic supergroups}\label{OSp intro} Let $V$ be a super vector space with $\dim(V_0) = m$, $m=2\ell$ or $2\ell+1$, $\dim(V_1) = 2n$, equipped with a non-degenerate supersymmetric bilinear form $(,)$. We denote the parity endomorphism of $V$ by $p$. The orthosymplectic Lie superalgebra $\mathfrak{osp}(m|2n)$ \cite[\S2.3]{Musson} consists of all endomorphisms of $V$ which respect the supersymmetric bilinear form \[ \mathfrak{osp}(m|2n) = \{ x \in \mathfrak{gl}(m|2n) \ | \ (xv,w) + (-1)^{p(x) p(v)} (v,xw) = 0 \ \forall v,w \in V_0 \cup V_1\}.\] It is simple for $m,n > 0$ and its even part is isomorphic to $\mathfrak{o}(m) \oplus \mathfrak{sp}(2n)$.

A super Harish-Chandra pair is a tuple $(G_0,\g)$ where $G_0$ is an algebraic affine group scheme and $\g$ is a finite-dimensional Lie superalgebra $\g = \g_0 \oplus \g_1$ which satisfies the following:
\begin{enumerate}
\item $\g_1$ is an algebraic representation of $G_0$
\item $\g_0 = Lie(G_0)$
\item The right adjoint action of $\g_0$ on $\g_1$ coincides with the action of $G_0$
\item The bracket $[,]: \g_1 \times \g_1 \to \g_0$ is $G_0$-equivariant where $\g_0$ is regarded as a right $G$-module via the adjoint action.
\end{enumerate}

By \cite{Masuoka}, \cite{Serganova-quasireductive} the category of super Harish-Chandra pairs is equivalent to the category of algebraic supergroups. Following Lehrer and Zhang \cite{Lehrer-Zhang-1} we define the orthosymplectic supergroup by the super Harish-Chandra pair $(O(m) \times Sp(2n), \mathfrak{osp}(m|2n))$. By definition, a finite dimensional representation $\rho$ of $\mathfrak{osp}(m|2n)$ defines a representation $\rho$ of $OSp(m|2n)$ if its restriction to $\mathfrak{osp}(m|2n)_0$ comes from an algebraic representation of $G_0$ \cite{Serganova-quasireductive}. Let $\Rep(G)$ denote the category of (algebraic) representations of $OSp(m|2n)$. Fix the morphism $\epsilon: \Z/2\Z \to G_0 = O(m) \times Sp(2n)$ which maps $-1$ to $diag(E_m,-E_{2n})$. We denote by $\Rep(G,\epsilon)$ the full subcategory of objects $(V,\rho)$ in $\Rep(G)$ such that $p_V = \rho(\epsilon)$ where $p$ denotes the parity morphism. Similarly we define the simple supergroup $SOSp(m|2n)$ by the super Harish-Chandra pair $(SO(m) \times Sp(2n), \mathfrak{osp}(m|2n))$ and we denote its representation category by $\Rep(G')$.

%The subcategory $Rep(G,\epsilon)$. define in general

Following Deligne \cite{Del07} we define for $\delta\in \Z$ the following triples $(G, \epsilon,X)$ where $G$ is a supergroup, $\epsilon$ an element of order 2 such that $int(\epsilon)$ induces on $\mathcal{O} (G)$ its grading modulo 2 and $X \in \Rep(G,\epsilon)$:

\begin{itemize}
 \item $\delta=m \geq 0: \ (O(m)=OSp(m|0), \id, V),$\footnote{For the case $\delta=0$, $O(0)$ is the trivial group, $V=0$, and $\Rep(G,\epsilon)$ is equivalent to the category of finite dimensional $\k$-vector spaces.}
 \item $\delta = -2n < 0: \ (Sp(2n)=OSp(0|2n),-\id, V \text{ seen as odd})$,
 \item $\delta = 1-2n < 0: \ (OSp(1|2n), diag(1,-1,\ldots,-1), V)$.
\end{itemize}

By the universal property \cite[Proposition 9.4]{Del07}, the assignment $st \mapsto X$ defines a tensor functor $F_\delta:\uRep(O_\delta) \to \Rep(G,\epsilon)$.

\begin{theorem} \label{quotient} \cite[Th\'{e}or\`{e}me 9.6]{Del07}. The tensor functor $F_\delta$ induces an equivalence of categories \[\uRep(O_\delta)/\NN \to \Rep(G,\epsilon)\] where $\NN$ denotes the tensor ideal of negligible morphisms.
\end{theorem}

By Corollary \ref{k=0}, an indecomposable element $R(\lambda)$ has dimension 0 if and only if $k(\lambda) = 0$. It is easy to show that the condition $k(\lambda) = 0$ is equivalent to the following in the three cases: 

\begin{itemize}
 \item $\delta=m \geq 0: \ \lambda_1^T+\lambda_2^T\leq m$,  %Hence, we always have   $F_m(R(\lambda))=\Sr_{[\lambda]}V$ in $\Rep(O(m))$ (see Proposition \ref{image for O} and Remark \ref{remark on O}).
 \item $\delta = -2n < 0: \ \lambda_1\leq n$,  %One can show $F_{-2n}(R(\lambda))=\Sr_{\<\lambda^T\>}V$ in $\Rep(Sp(2n))$ where $\Sr_{\<\lambda^T\>}V$ is defined in \cite[\S17.3]{FH}.
 \item $\delta = 1-2n < 0: \ \lambda_1+\lambda_2\leq 1+2n$.
 \end{itemize}
In particular, this shows $F_m(R(\lambda))=\Sr_{[\lambda]}V$ in $\Rep(O(m))$ for all $\lambda$ (see Proposition \ref{image for O} and Remark \ref{remark on O}).   One can similarly show $F_{-2n}(R(\lambda))=\Sr_{\<\lambda^T\>}V$ in $\Rep(Sp(2n))$ where $\Sr_{\<\lambda^T\>}V$ is defined in \cite[\S17.3]{FH}.

More generally, by the universal property the assignment $st \mapsto V = \k^{m|2n}$ defines a tensor functor $F_{m|2n}:\uRep(O_{m-2n}) \to \Rep(OSp(m|2n))$.  As in \S\ref{intro: OSp} we denote $\I(m|2n)=ker(F_{m|2n})$. We also use the notation $R_{m|2n}(\lambda) = F_{m|2n}(R(\lambda))$.  The next proposition is from \cite[Corollary 5.8]{Lehrer-Zhang-1}.

\begin{proposition} \label{thm:LZ-fullness} $\End_{OSp(m|2n)} (V^{\otimes r}) = F_{m|2n}(B_r(m-2n))$.
\end{proposition}
 
\begin{theorem}\label{full} i) The functor $F_{m|2n}$ is full.

ii) The set $\{  R_{m|2n}(\lambda) \ | \ \lambda\in\Lambda_r, R_{m|2n}(\lambda) \neq 0 \}$ is a complete set of pairwise non-isomorphic indecomposable summands of the tensor power $V^{\otimes r}$. 
%The elements in this set are pairwise non-isomorphic.
\end{theorem}

\begin{proof} (i) Since $F_{m|2n}$ is additive it suffices to show $F_{m|2n}$ restricted to $\B(\delta)$ is full.  Since the central element $-\id\in OSp(m|2n)$ acts on $V^{\otimes r}$ by $(-1)^r$, it follows that $\Hom_{OSp(m|2n)}(V^{\otimes r},V^{\otimes s})=0$ unless $r$ and $s$ have the same parity.  Thus, it suffices to show $F_{m|2n}$ induces a surjection $\Hom_{\B(\delta)}(r,s)\to\Hom_{OSp(m|2n)}(V^{\otimes r},V^{\otimes s})$ whenever $r$ and $s$ have the same parity.  Now, in any $\k$-linear rigid symmetric monoidal category where objects are self dual there is an isomorphism of $\k$-vector spaces $\Hom(X^{\otimes r},X^{\otimes s})\to\Hom(X^{\otimes r-1},X^{\otimes s+1})$ which is compatible with any tensor functor.  Hence, the result follows from Proposition \ref{thm:LZ-fullness}.

(ii) By \cite[Proposition 2.7.4]{CW}, $F_{m|2n}(X)$ is indecomposable whenever $X$ is indecomposable and $F_{m|2n}(X) \simeq F_{m|2n}(Y)$ if and only if $X \simeq Y$. An indecomposable summand $R$ in $V^{\otimes r}$ corresponds to an idempotent $e\in \End_{OSp(m|2n)}(V^{\otimes r})$. Since $F_{m|2n}$ is full, $e$ has a preimage $\tilde{e}\in B_r(m-2n)$. Now, if  $\im(\tilde{e})$ decomposes in $\uRep(O_\delta)$ as $ \im(\tilde{e})= R(\lambda^{(1)}) \oplus \cdots\oplus R(\lambda^{(k)})$, then $R=R_{m|2n} (\lambda^{(1)}) \oplus  \cdots \oplus R_{m|2n} (\lambda^{(k)})$. Since $R$ is indecomposable, we must have $R=R_{m|2n}(\lambda^{(i)})$ for one $i \in \{1,\ldots,k\}$.
%see cw 2.7.4 and cw 4.7.1
\end{proof}

%We denote the full subcategory of $Rep(OSp(m|2n))$ of direct summands in some tensor power of the standard representation by $T(m|2n)$. In particular, for each $m,n\geq0$, $T(m|2n)$ contains the object $\k=V^{\otimes 0}$.  

\begin{remark}\label{B not onto osp} By \cite{Lehrer-Zhang-1} and \cite{Stroppel-Ehrig-brauer} the Brauer algebra does not always map surjectively onto  $\End_{\mathfrak{osp}(m|2n)}(V^{\otimes r})$. This is linked with the existence of the super-Pfaffian as explained in \cite[\S7]{Lehrer-Zhang-1}.
\end{remark}

By Theorem \ref{full} the image of $F_{m|2n}$, which we denote $T(m|2n)$, is a full subcategory of $\Rep(OSp(m|2n))$.  The objects in $T(m|2n)$ are direct sums of the $R_{m|2n}(\lambda)$'s.  

%\medskip

%here i used without further checking results of lehrer and zhang on the brauer category over the grassmannian. it should be checked that these results hold over $k$.

%\medskip

A representation $X$ of a supergroup $G$ is a tensor generator if every representation is a subquotient of a finite direct sum of representations $X^{\otimes r} \otimes (X^{\vee})^{\otimes s}$ and their parity shifts for some $r,s \geq 0$. In the $\Rep(OSp(m|2n))$-case it is easily seen that $X =V$ is  a tensor generator as in \cite{Heidersdorf-mixed-tensors}. In fact every faithful representation of an algebraic supergroup is a tensor generator as in the classical case \cite[Proposition 2.20]{Deligne-Milne} \cite[\S3.5]{Waterhouse}. The proof is virtually the same as in the classical case \cite[Proposition 3.1]{Deligne-hodge}. For the convenience of the reader we recall the argument. Let $G$ be an algebraic supergroup and $\rho: G \to Gl(V)$, $V \cong \k^{m|n}$, a faithful representation. Then $\rho$ induces a surjection $\k[Gl(V)] \to \k[G]$ of super Hopf algebras. This can be proven as for algebraic groups, but also follows from the splitting theorem \cite[Theorem 4.5]{Masuoka-fundamental} \cite{Weissauer-semisimple}: The super Hopf algebra decomposes as \[ \k[G] \cong \k[G_0] \otimes_{\k} \k[\theta_1,\ldots,\theta_s] \] where $\theta_1,\ldots,\theta_s$ are a $\k$-basis of $(\mathfrak{g}_1)^*$, the dual of the odd part of the Lie superalgebra attached to $G$. The induced morphism of algebraic groups $G_0 \to Gl(m|n)_0$ gives a surjection $\k[Gl(m) \times Gl(n)] \to \k[G_0]$ by \cite{Deligne-Milne}. The homomorphism of Lie superalgebras $Lie(G) \to Lie(Gl(V))$ obtained from $\rho$ is injective and in turn gives a surjection $(\mathfrak{gl}(V)_1)^* \to (\mathfrak{g}_1)^*$. The two surjections define the surjection $\k[Gl(V)] \to \k[G]$. By \cite[Proposition 9.3.1]{Westra} every finite-dimensional representation $V$ of $G$ is a submodule of $\k[G]^{\dim V}$. Hence it suffices to prove the statement for the regular representation. By definition of the supersymmetric algebra we have a surjective map \[  (V \otimes V^{\vee})^{\otimes \bullet} \otimes (V^{\vee} \otimes V)^{\otimes \bullet} \to Sym^{\otimes \bullet} (End(V)) \otimes Sym^{\bullet} (End(V^{\vee})). \] The immersion $Gl(V) \to End(V) \times End(V^{\vee})$ defined by $g \mapsto (g,(g^{-1})^{\vee})$ maps $Gl(V)$ onto $\{ (g,h) \in End(V) \times End(V^{\vee}) \ | \ g \circ h^{\vee} = id_V \}$. Therefore we get a third surjection \[ Sym^{\otimes \bullet} (End(V)) \otimes Sym^{\bullet} (End(V^{\vee})) \to \k[Gl(V)]\] which combined with $\k[Gl(V)] \to \k[G]$ proves the statement for the regular representation.

\begin{lemma} Every projective representation of $\Rep(OSp(m|2n))$ is in $T(m|2n)$.
\end{lemma}

\begin{proof} The module $V$ is a tensor generator of $\Rep(OSp(m|2n))$. Hence every $OSp(m|2n)$-module appears as a subquotient of some direct sum of iterated tensor products $V^{\otimes r}$. Projectives and injectives coincide \cite[Proposition 2.2.2]{BKN-complexity} in $\Rep(G)$. If $P$ is an indecomposable projective representation appearing as a quotient of some submodule $M$, then it is already a direct summand of the submodule (projectivity) and therefore already a direct summand (injectivity).
\end{proof}

We remark that the indecomposable projective modules in $\Rep(OSp(m|2n))$ are the irreducible typical modules and the projective covers of the irreducible atypical modules.

%%%%%%%%%%%%%%%%%%%%%%%%%%%%%%%%%%%

\subsection{Cohomological tensor functors}\label{cohomological}

Following Duflo-Serganova \cite{Duflo-Serganova} (see also \cite[\S6.1]{Serganova14}) we denote by $X$ the cone of self-commuting elements \[ X = \{ x \in \g_1 \ | \ [x,x] = 0\}.\] 
For  $x \in X$ there exist $g \in G_0$ and isotropic mutually orthogonal linearly independent roots $\alpha_1, \ldots, \alpha_r$ such that $Ad_g(x) = x_1 + \cdots + x_r$ with each $x_i \in \g_{\alpha_i}$. The number $r$ is called the rank of $x$. The rank is a number between 1 and $def(\g) =\min(\ell,n)$. For  $x \in X$ of rank $r$ and any representation $(M,\rho)  \in \Rep(G)$ we define $F_x(M) = ker(\rho)/Im(\rho)$. As in \emph{loc.~cit.}, this defines a tensor functor $F_x: \Rep(OSp(m|2n)) \to \Rep(OSp(m-2r|2n-2r))$. For $M \in \Rep(G)$ we define the associated variety by \[X_M = \{ x \in X \ | \ F_x (M) \neq 0\}. \]  

%is the following necessary?
%\begin{lem} $DS$ induces a tensor functor $DS: Rep(OSp(m|2n)) \to Rep(OSp(m-2|2n-2)$.
%\end{lem}

From now on we will study the Duflo-Serganova tensor functor associated to \[ x = \begin{pmatrix} 0 & \epsilon \\ 0 & 0 \end{pmatrix}, \ \ \epsilon = diag(1,0,\ldots,0) \] and denote the corresponding tensor functor by $DS$. An easy computation verifies the next lemma.

\begin{lemma}\label{ker(DS)} $DS$ maps $V$ to the standard representation of $OSp(m-2|2n-2)$.
\end{lemma}

\begin{lemma}\label{ker DS} The kernel of $DS$ is equal to $Proj$.
\end{lemma}

%For $x \in X$ we denote by $U<x>$ the Lie superalgebra. The condition $F_x(M) = 0$ for $x \in X$ is equivalent to $M$ being projective as a $U<x>$-module. Let $\{ x_i \ | \ i \in I\}$ be a set of orbit representatives for the minimal orbits of the action of $G_0$ on $\g_1$. 

\begin{proof} It is sufficient to test this for $G'$. By \cite[Theorem 3.4]{Duflo-Serganova}, $M$ is projective if and only if $X_M = \{ 0\}$.  The variety $X_M$ is a Zariski-closed $G'_0$-invariant subvariety of $X$ by \cite{Duflo-Serganova}. Then $X_M$ is non-zero if and only if it contains a minimal $G'_0$-orbit with respect to the partial order given by containment in closures. The computations in the proof of  \cite[Theorem 4.2]{Duflo-Serganova} show that the set $X_1 = \{ x \in X \ | \ rk(x) =1 \}$ is the only nontrivial minimal orbit for the $G'_0$-action on $X$. Since our fixed $x$ is in $X_1$, we conclude that $DS(M) = 0$ implies that $X_M = \{0\}$ and hence that $M$ is projective.
\end{proof}

\begin{corollary}\label{image} Under $DS$  \begin{align*} R_{m|2n}(\lambda) \mapsto \begin{cases} 0,  & \text{if }R_{m|2n}(\lambda) \text{ is projective}; \\ R_{m-2|2n-2}(\lambda), &  \text{otherwise.} \end{cases} \end{align*}  
\end{corollary}

\begin{proof} This follows from the diagram 
\begin{equation}\label{F DS} \xymatrix{ \uRep(O_{m-2n}) \ar[d]^{F_{m|2n}} \ar[dr]^{F_{m-2|2n-2}} & \\ \Rep(OSp(m|2n)) \ar[r]^{DS\qquad} & \Rep(OSp(m-2|2n-2)). } \end{equation}
 Since $DS$ maps $V$ to the standard representation, the universal property of Deligne's category \cite[Proposition 9.4]{Del07} implies that the diagram is commutative.  
\end{proof}

\begin{corollary}\label{DS T} $DS$ restricts to a functor $DS: T(m|2n) \to T(m-2|2n-2)$.
\end{corollary}

\begin{corollary} \label{thm:inclusions} We have strict inclusions $\ldots \subsetneq \I(m|2n) \subsetneq \I(m-2|2n-2) \subsetneq \ldots$.
\end{corollary}

\begin{proof} The kernel of $F_{m-2|2n-2}$ is the kernel of the composed functor $DS \circ F_{m|2n}$. The kernel of $DS$ is $Proj$ which is nontrivial.
\end{proof} 

\subsection{Another classification of thick ideals} In this subsection we show that every nonzero proper thick ideal in $\uRep(O_\delta)$ is of the form $\I(m|2n)=ker(F_{m|2n})$ for some $m,n\in\Z_{\geq0}$ with $m-2n = \delta$.

\begin{lemma}\label{proj-minimal} Let $\mathcal{C}$ be a pseudoabelian $\k$-linear rigid symmetric monoidal category. Then any thick ideal $\mathcal{I}$ contains the thick ideal $Proj$.
\end{lemma}

\begin{proof} Let $P$ be projective and $M \in \mathcal{I}$. Then $\End(M) = \Hom(M \otimes M^{\vee},\one)$. In particular we have an epimorphism $\varphi: M \otimes M^{\vee} \to \one$. The composite  \[ \xymatrix{ P \otimes M \otimes M^{\vee} \ar[r]^{\quad  \id \otimes \varphi} & P \otimes \one \ar[r]^{\sim} & P } \] is an epimorphism. Since $P$ is projective it has  a right-inverse. Hence $P$ is a retract of the element $P \otimes M \otimes M^{\vee}$ and hence in $\mathcal{I}$.
\end{proof}

\begin{comment}
Let $\mathcal{I}$ be a thick ideal in $\uRep(O_\delta)$. Then $\mathcal{I} \subset \mathcal{N}$. Further it is clear that $\mathcal{I}$ contains an ideal of the form $\I(m'|2n')$. Hence we can assume that there exists pairs of integers $(m_1,n_1)$ and $(m_2,n_2)$ such that \[ \I(m_1|2n_1) \subset \mathcal{I} \subset \I(m_2|2n_2).\] Since $F_{m|2n}: \uRep(O_\delta) \to T(m|2n)$ is surjective, the image of an ideal in $T(m|2n)$ is again an ideal. For $m =2n$ we use implicitely the full subcategory $T(m|2n) \cup \{ \k \}$.
\end{comment}

\begin{theorem} \label{theorem:kernel} Every nonzero proper thick ideal in $\uRep(O_\delta)$ is of the form $\I(m|2n)$.% = ker(F_{m|2n})$.
\end{theorem}

\begin{proof}  
Let $\I$ be a nonzero proper thick ideal in $\uRep(O_\delta)$.  By Theorem \ref{I_k theorem} $\I$ belongs to the chain $\I_1\supseteq \I_2\supseteq \I_3\supseteq \cdots$, hence by Corollary \ref{thm:inclusions} we can choose $m\in\Z_{\geq0}$ which is minimal with $\I(m|2n)\subset\I$ and $\delta=m-2n$.   The image of $\mathcal{I}$ under $F_{m|2n}$, denoted $\I'$, is a thick ideal in $T(m|2n)$.  If $\I'$ is nonzero, then it contains $Proj\subset T(m|2n)$ by Lemma \ref{proj-minimal}.  In this case it follows from Lemma \ref{ker(DS)} and Corollary \ref{DS T} that  $\I'$ contains the kernel of $DS: T(m|2n) \to T(m -2|2n -2)$, and thus by the commutative diagram (\ref{F DS}) we have $\I(m-2|2n-2)\subset\I$, which contradicts the minimality of $m$.  Therefore $\I'=0$, which implies $\I=\I(m|2n)$.
\end{proof}

\begin{corollary}\label{kernel}   $\I(m|2n)=\I_{\min(\ell,n)+1}$.
%$R(\lambda)$ is in the kernel of $F_{m|2n}: \uRep(O_{m-2n}) \to \Rep(OSp(m|2n))$ if and only if $k(\lambda) \geq \min(\ell,n) + 1$. 
\end{corollary}

\begin{corollary}\label{tensors} The assignment $\lambda\mapsto R_{m|2n}(\lambda)$ is a bijection from the set of all $\lambda\in\Lambda_r$ with $k(\lambda)\leq\min(\ell,n)$ to the set of isomorphism classes of nonzero indecomposable summands of  $V^{\otimes r}$ in $\Rep(OSp(m|2n))$.
\end{corollary}

\begin{remark}
We work with $OSp$ instead of $SOSp$ since the analogue of $F_{m|2n}$ for $SOSp$ is not full (see Remark \ref{B not onto osp}). The restriction functor from $\Rep(OSp(m|2n))$ to $\Rep(SOSp(m|2n))$ sends the standard representation to the standard representation. Because a tensor functor from the Deligne category is uniquely determined by the image of $st$, the two tensor functors \[ F'_{m|2n}: \uRep(O_{m-2n}) \to \Rep(SOSp(m|2n)), \]  \[ Res \circ F_{m|2n}: \uRep(O_{m-2n}) \to \Rep(OSp(m|2n)) \to \Rep(SOSp(m|2n)) \] are isomorphic. The restriction of irreducible and projective representations is determined in \cite[\S2.2]{Ehrig-Stroppel-osp}. Note in particular that $F'_{m|2n}(R(\lambda))$ is in general not indecomposable. This can be seen even in the classical case when $n=0$ and $m=2\ell$. Indeed, $F_{2\ell}(R(\lambda))=\Sr_{[\lambda]}V$ restricts to the sum of two simple $SO(2\ell)$-modules whenever $\lambda$ has length $\ell$ \cite[Theorem 19.22]{FH}.
\end{remark}

%Let \begin{align*} \Lambda_r = \begin{cases}  \{ \lambda \in \Lambda \ | \ |\lambda| = r - 2i, \ 0 \leq i \leq \frac{r}{2} \} \text{ if } \delta \neq 0 \text{ or } r \text{ odd.} \\  \{ \lambda \in \Lambda \ | \ |\lambda| = r - 2i, \ 0 \leq i < \frac{r}{2} \} \text{ if } \delta = 0 \text{ and } r \text{ even.} \end{cases} \end{align*}

%We denote by $\psi_r^{m|2n}: B_r(m-2n) \to End(V^{\otimes r})$ the surjective map induced by $F_{m|2n}$. 
%
%\begin{corollary} The modules \[ \{ D_r(\lambda) \ | \ \text{ for all } \lambda \in  \Lambda_r \text{ with } k(\lambda) \leq min(m,n) \} \] give a complete set of pairwise non-isomorphic irreducible $B_r(m-2n)/ker(\psi_r^{m|2n})$-modules. 
%\end{corollary}

 %%%%%%%%%%%%%%%%%%%%%%%%%%%%%%%%%

\subsection{Tensors of $OSp(m|2n)$} This subsection consists of a few immediate applications concerning tensors of $\Rep(OSp(m|2n))$.  

\begin{lemma} A tensor $R_{m|2n}(\lambda)$ in $\Rep(OSp(m|2n))$ is projective if and only if $k(\lambda) = \min(\ell,n)$.
\end{lemma}

\begin{comment}
\begin{proof} The commutative diagram of Theorem \ref{thm:inclusions} shows that the kernel of the functor $F_{m|2n} \circ DS$ equals the kernel of $F_{m-2|2n-2}$.
\end{proof}
\end{comment}

\begin{proof}
This follows from Lemma \ref{ker DS} and the commutative diagram (\ref{F DS}).
\end{proof}

The lemma provides no way to read of the weight of the socle or head of the projective representation. Such a description is given in \cite[Proposition H]{Ehrig-Stroppel-osp}.

\begin{corollary} The superdimension $sdim(R_{m|2n}(\lambda))\not=0$  if and only if $k(\lambda) = 0$. Moreover, $sdim(R_{m|2n}(\lambda))=sdim(F_{m-2n}(R(\lambda)))$ for all $\lambda$.  
\end{corollary} 

\begin{comment}
\begin{corollary} A tensor $R(\lambda)$ has non-vanishing superdimension if and only if $k(\lambda) = 0$. If $k(\lambda) = 0$, then $sdim R(\lambda)$ is given by $sdim(F_d(R(\lambda)))$. 
\end{corollary} 
\end{comment}

\begin{lemma} Every atypical block contains a non-projective tensor $R_{m|2n}(\lambda)$.
\end{lemma}

\begin{proof} Every typical representation is of the form $R(\lambda)$ for some partition $\lambda$. Let $\Gamma$ be a block of atypicality $k$ and choose $x \in X$ with $\rk(x) = k$ such that the functor $F_x: \Rep(OSp(m|2n)) \to  \Rep(OSp(m-2k|2n-2k))$ sends the standard representation to the standard representation. The block $\Gamma$ is determined by its core $L^{core}$ as in \cite{Gruson-Serganova}, an irreducible typical representation in $\Rep(OSp(m-2k|2n-2k))$. Then $L^{core} = R_{m-2k|2n-2k}(\lambda)$ for some $\lambda$. Since $F_x(R_{m|2n}(\lambda)) = R_{m-2k|2n-2k}(\lambda)$, the module $R_{m|2n}(\lambda) \in \Rep(OSp(m|2n))$ is in $\Gamma$. It is not projective since it is not in the kernel of $F_x$. 
\end{proof}

Applying $DS$ $\min(\ell,n)$-times gives a functor $\Rep(OSp(m|2n))\to\Rep(G,\epsilon)$.  Equivalently, we could pick any $x \in X$ with $\rk(x) = \min(\ell,n)$ such that the functor $F_x:\Rep(OSp(m|2n))\to\Rep(G,\epsilon)$ maps $V\mapsto X$.  Then the functor $F_x$ is isomorphic to $DS^{min(l,n)}$ on $T$. By Lemma \ref{ker(DS)} and the universal property \cite[Proposition 9.4]{Del07} we have $F_x\circ F_{m|2n}=F_{m-2n}$.  Restricting to $T(m|2n)$ gives us the following commutative diagram of tensor functors 
\[ \xymatrix{ & \uRep(O_{m-2n}) \ar[dd]^{F_{m-2n}} \ar[dl]_{F_{m|2n}} \\ T(m|2n) \ar[dr]^{F_x} & \\ & \Rep(G,\epsilon). } \] 

\begin{theorem}  $T(m|2n)/\mathcal{N} \simeq \Rep(G,\epsilon)$.
\end{theorem}

\begin{proof} Since $F_{m-2n}$ is full (Theorem \ref{full}), the restriction of $F_x$ to $T(m|2n)$ is also full, and hence  factors through $T(m|2n)/\mathcal{N}$. By Theorem \ref{quotient}, $F_{m-2n}$  gives  a  bijection between the simple objects of $\Rep(G,\epsilon)$ and the indecomposable objects $R$ in $\uRep(O_{m-2n})$ with $\id_R \notin \mathcal{N}$. Any $R$ in $\uRep(O_{m-2n})$ with $\id_R \in \mathcal{N}$ maps to zero in $T(m|2n)/\mathcal{N}$. Note that the image of an indecomposable element of $\uRep(O_{m-2n})$ in $T(m|2n)/\mathcal{N}$ is indecomposable since $F_{m|2n}$ is full. This shows that the functor $T(m|2n)/\mathcal{N} \to \Rep(G,\epsilon)$ is one-to-one on objects. Fully faithfullness follows trivially from Schur's lemma. 
\end{proof}

\section{Related questions and open problems}\label{questions}

%Question: If one knows the image in $Rep(G,\epsilon)$, can one determine $F_x(R(\lambda))$.
%
%to do: Einfuegen: Die Bedingung $k(\lambda) = 0$ ist aequivalent zu den klassischen Parametrisierungen der irreduziblen. Comes.

%\textbf{Example.} Let us consider the $osp(1|2)$-case.

%%%%%%%%%%%%%%%%%%%%%%%%%%%%%%%%%%%%%%%%%%%%%%

\subsection{Orthosymplectic supergroups} 1) In the $Gl(m|n)$-case ($m \geq n$) the atypicality of $R(\lambda)$ is given by $n - \rk(\lambda)$. The Loewy length of $R(\lambda)$ is equal to $2 \df(\lambda) + 1$ and its socle is irreducible. One could expect that this holds true in the orthosymplectic case as well.

\medskip

2) More generally it would be interesting to understand the Loewy filtrations of the $R_{m|2n}(\lambda)$ and compute $F_{m|2n}(R(\lambda))$ as in the $Gl(m|n)$-case \cite{Heidersdorf-mixed-tensors}. This would imply tensor product decompositions for the representations in $T(m|2n)$. We plan to address these questions in a second paper.

\medskip

3) The summands of the tensor power $V^{\otimes r}$ have been studied in \cite{Benkart-Lee-Shader-Ram} for $|m-n| > r$. They show that the character of their summands $T^{\lambda}$ is given by a function $sc_{\lambda}$ which can be written as a linear combination of Schur functions $s_{\mu}$ \cite[Theorem 4.24]{Benkart-Lee-Shader-Ram}. It is not obvious how to extend this character formula inductively to larger $r$ as in \cite[Theorem 8.5.2]{CW}.
If we want to compute $ch(R(\lambda))$, we might induct on the degree and the lexicographic ordering of the partitions and assume a formula for the character of $R(\mu)$ in terms of  $sc_{\mu}$ for $\mu$ smaller than $\lambda$ (note that \cite{Benkart-Lee-Shader-Ram} work actually with $\mathfrak{spo}(m|n)$ with $m$ even, hence this is an abuse of notation). Write $\lambda = \mu + \nu$ for two partitions $\mu$ and $\nu$. By the results on the tensor product decomposition we can write $R(\mu) \otimes R(\nu) = R(\mu + \nu) + \bigoplus_i R(\eta_i)$ where the $\eta_i$ are partitions which are smaller than $\mu + \nu$, either in degree or lexicographic ordering. If we take the character on both sides we would need a nice formula for $sc_{\mu} sc_{\nu}$ which does not seem to be obvious.

\medskip

4) In \cite[Corollary 5.8]{Lehrer-Zhang-2}, Lehrer and Zhang describe the kernel of the surjective map $\psi_r^{m|2n}: B_r(m-2n) \to \End_{OSp(m|2n)}(V^{\otimes r})$  induced by $F_{m|2n}$. Is there a more direct description using the $k(\lambda)$-condition? 

%should be: ideal in $B_r(m-2n)$ generated by the partitions satisfying $k(\lambda) \geq min(m,n) + 1$.
%2) Is $F_{m|2n}$ a quotient functor.
%
%3) Endomorphism-spaces
%
%4) Relation to $T(\lambda)$ of Benkart et al.

%%%%%%%%%%%%%%%%%%%%%%%%%%%%%%%%%%%%%%%%%%%%%%

%%%%%%%%%%%%%%%%%%%%%%%%%%%%%%%%%%%%%%%%%%%%%%%%

\subsection{Strange supergroups} An odd Brauer supercategory has been introduced by \cite{KT} and \cite{Serganova14} (see also \cite{BE}).  This supercategory plays the role in the representation theory of the pareiplectic supergroup $P(n)$ analogous to that of the Brauer category for orthosymplectic supergroups.  Similarly, for the queer supergroup $Q(n)$ one can define an oriented Brauer-Clifford supercategory whose endomorphism algebras are the walled Brauer superalgebras from \cite{JK}. Taking super-Karoubi envelopes of these diagram supercategories yield supercategories $\uRep(P)$ and $\uRep(Q)$ which admit functors to $\Rep(P(n))$ and $\Rep(Q(n))$ respectively.  
Perhaps one could classify thick ideals in $\uRep(P)$ (resp.~$\uRep(Q)$) and use that classification to better understand the indecomposable summands of tensor powers (resp.~mixed tensor powers) of the standard representation of $P(n)$ (resp.~$Q(n)$).  

\subsection{Modified traces} Thick ideals will sometimes admit modified trace and dimension functions as defined in \cite{GKP}.   For example, in \cite{CK} it was shown that if $\delta\in\Z_{\geq0}$ then the unique nonzero proper thick ideal in Deligne's $\uRep(S_\delta)$  admits a modified trace.  It would be interesting to determine which thick ideals in $\uRep(O_\delta)$ admit modified traces. 

\renewcommand{\bibname}{\textsc{references}} 
\bibliographystyle{alphanum}	
\bibliography{references}	

\def\cprime{$'$} \def\cprime{$'$}
\begin{thebibliography}{{Mas}2}

\bibitem[AK]{Andre-Kahn}
Y.~{Andr\'e} and B.~{Kahn}.
\newblock {Nilpotence, radicaux et structures mono\"\i dales.}
\newblock {\em {Rend. Semin. Mat. Univ. Padova}}, 108:107--291, 2002.

\bibitem[BE]{BE}
J.~{Brundan} and A.~{Ellis}.
\newblock {Monoidal supercategories}.
\newblock {\em ArXiv e-prints}, March 2016.

\bibitem[BKN]{BKN-complexity}
B.~{Boe}, J.~{Kujawa}, and D.~{Nakano}.
\newblock {Complexity and module varieties for classical Lie superalgebras.}
\newblock {\em {Int. Math. Res. Not.}}, 2011(3):696--724, 2011.

\bibitem[BLR]{Benkart-Lee-Shader-Ram}
G.~{Benkart}, C.~{Lee Shader}, and A.~{Ram}.
\newblock {Tensor product representations for orthosymplectic Lie
  superalgebras.}
\newblock {\em {J. Pure Appl. Algebra}}, 130(1):1--48, 1998.

\bibitem[BR]{BR}
A.~Berele and A.~Regev.
\newblock Hook {Y}oung diagrams with applications to combinatorics and to
  representations of {L}ie superalgebras.
\newblock {\em Adv. in Math.}, 64(2):118--175, 1987.

\bibitem[Bra]{Brauer}
R.~Brauer.
\newblock On algebras which are connected with the semisimple continuous
  groups.
\newblock {\em Ann. of Math. (2)}, 38(4):857--872, 1937.

\bibitem[BS1]{BS2}
J.~{Brundan} and C.~{Stroppel}.
\newblock Highest weight categories arising from {K}hovanov's diagram algebra
  {II}: {K}oszulity.
\newblock {\em Transform. Groups}, 15:1--45, 2010.

\bibitem[BS2]{BS1}
J.~{Brundan} and C.~{Stroppel}.
\newblock Highest weight categories arising from {K}hovanov's diagram algebra
  {I}: cellularity.
\newblock {\em Mosc. Math. J.}, 11:685 -- 722, 2011.

\bibitem[BS3]{BS3}
J.~{Brundan} and C.~{Stroppel}.
\newblock Highest weight categories arising from {K}hovanov's diagram algebra
  {III}: category {O}.
\newblock {\em Represent. Theory}, 15:170 -- 243, 2011.

\bibitem[BS4]{BS}
J.~{Brundan} and C.~{Stroppel}.
\newblock {Gradings on walled Brauer algebras and Khovanov's arc algebra}.
\newblock {\em Adv. in Math.}, 231(2):709 -- 773, 2012.

\bibitem[BS5]{BS4}
J.~{Brundan} and C.~{Stroppel}.
\newblock Highest weight categories arising from {K}hovanov's diagram algebra
  {IV}: the general linear supergroup.
\newblock {\em J. Eur. Math. Soc.}, 14:373 -- 419, 2012.

\bibitem[CD]{CD}
A.~Cox and M.~{De Visscher}.
\newblock Diagrammatic {K}azhdan-{L}usztig theory for the (walled) {B}rauer
  algebra.
\newblock {\em J. Algebra}, 340(1):151 -- 181, 2011.

\bibitem[CDM]{CDM}
A.~Cox, M.~{De Visscher}, and P.~Martin.
\newblock The blocks of the {B}rauer algebra in characteristic zero.
\newblock {\em Represent. Theory}, 13:272--308, 2009.

\bibitem[CK]{CK}
J.~{Comes} and J.~{Kujawa}.
\newblock Modified traces on {D}eligne's category
  {$\underline{\operatorname{Re}}\!\operatorname{p}(S_t)$}.
\newblock {\em Journal of Algebraic Combinatorics}, 36:541 -- 560, 2012.

\bibitem[CO]{CO11}
J.~Comes and V.~Ostrik.
\newblock On blocks of {D}eligne's category
  {$\underline{\operatorname{Re}}\!\operatorname{p}(S_t)$}.
\newblock {\em Adv. in Math.}, 226(2):1331--1377, 2011.

\bibitem[Com]{Comes}
J.~Comes.
\newblock Ideals in {D}eligne's category
  {$\underline{\operatorname{Re}}\!\operatorname{p}(GL_\delta)$}.
\newblock {\em Mathematical Research Letters}, 21(5):969--984, 2014.

\bibitem[CW]{CW}
J.~Comes and B.~Wilson.
\newblock {{D}eligne's category
  {$\underline{\operatorname{Re}}\!\operatorname{p}(GL_\delta)$} and
  representations of general linear supergroups}.
\newblock {\em Represent. Theory}, 16:568--609, 2012.

\bibitem[{Del}1]{Deligne-hodge}
P.~{Deligne}.
\newblock {Hodge cycles on abelian varieties. (Notes by J. S. Milne).}
\newblock {Hodge cycles, motives, and Shimura varieties, Lect. Notes Math. 900,
  9-100}, 1982.

\bibitem[Del2]{Del96}
P.~Deligne.
\newblock La s\'erie exceptionnelle de groupes de {L}ie.
\newblock {\em C. R. Acad. Sci. Paris S\'er. I Math.}, 322(4):321--326, 1996.

\bibitem[Del3]{Del07}
P.~Deligne.
\newblock La cat\'egorie des repr\'esentations du groupe sym\'etrique {$S\sb
  t$}, lorsque {$t$} n'est pas un entier naturel.
\newblock In {\em Algebraic groups and homogeneous spaces}, Tata Inst. Fund.
  Res. Stud. Math., pages 209--273. Tata Inst. Fund. Res., Mumbai, 2007.

\bibitem[DM]{Deligne-Milne}
P.~{Deligne} and J.S. {Milne}.
\newblock {Tannakian categories.}
\newblock {Hodge cycles, motives, and Shimura varieties, Lect. Notes Math. 900,
  101-228}, 1982.

\bibitem[DS]{Duflo-Serganova}
M.~Duflo and V.~Serganova.
\newblock On associated variety for {L}ie superalgebras.
\newblock {\em ArXiv e-prints}, July 2005.

\bibitem[ES1]{Ehrig-Stroppel}
M.~Ehrig and C.~Stroppel.
\newblock Diagrammatic description for the categories of perverse sheaves on
  isotropic grassmannians.
\newblock {\em Selecta Mathematica}, 22(3):1455--1536, 2016.

\bibitem[ES2]{Ehrig-Stroppel-osp}
M.~Ehrig and C.~Stroppel.
\newblock {On the category of finite-dimensional representations of $OSP(r|n)$:
  Part I}.
\newblock {\em Representation theory - current trends and perspectives, EMS
  Series of Congress Reports, European Mathematical Society (EMS)}, 2016.

\bibitem[ES3]{Stroppel-Ehrig-brauer}
M.~Ehrig and C.~Stroppel.
\newblock {Schur--Weyl duality for the Brauer algebra and the ortho-symplectic
  Lie superalgebra}.
\newblock {\em Mathematische Zeitschrift}, 284(1):595--613, 2016.

\bibitem[FH]{FH}
W.~Fulton and J.~Harris.
\newblock {\em Representation theory}, volume 129 of {\em Graduate Texts in
  Mathematics}.
\newblock Springer-Verlag, New York, 1991.
\newblock A first course, Readings in Mathematics.

\bibitem[GKP]{GKP}
N.~{Geer}, J.~{Kujawa}, and B.~{Patureau-Mirand}.
\newblock {Generalized trace and modified dimension functions on ribbon
  categories}.
\newblock {\em Selecta Math.}, 17(2):453--504, 2011.

\bibitem[GL]{GL}
J.~Graham and G.~Lehrer.
\newblock Cellular algebras.
\newblock {\em Invent. Math.}, 123(1):1--34, 1996.

\bibitem[GS]{Gruson-Serganova}
C.~{Gruson} and V.~{Serganova}.
\newblock {Cohomology of generalized supergrassmannians and character formulae
  for basic classical Lie superalgebras.}
\newblock {\em {Proc. Lond. Math. Soc. (3)}}, 101(3):852--892, 2010.

\bibitem[{Hei}]{Heidersdorf-mixed-tensors}
T.~{Heidersdorf}.
\newblock Mixed tensors of the general linear supergroup.
\newblock {\em ArXiv e-prints}, June 2014.

\bibitem[JK]{JK}
J.~Jung and S.~Kang.
\newblock {Mixed Schur-Weyl-Sergeev duality for queer Lie superalgebras}.
\newblock {\em Journal of Algebra}, 399:516 -- 545, 2014.

\bibitem[Koi]{Koike}
K.~Koike.
\newblock On the decomposition of tensor products of the representations of the
  classical groups: by means of the universal characters.
\newblock {\em Adv. Math.}, 74(1):57--86, 1989.

\bibitem[KT]{KT}
J.~Kujawa and B.~Tharp.
\newblock The marked {B}rauer category.
\newblock {\em ArXiv e-prints}, November 2014.

\bibitem[LS]{Lejczyk-Stroppel}
T.~Lejczyk and C.~Stroppel.
\newblock A graphical description of {$(D_n,A_{n-1})$} {K}azhdan--{L}usztig
  polynomials.
\newblock {\em Glasgow Mathematical Journal}, 55(2):313--340, 2012.

\bibitem[LZ1]{Lehrer-Zhang-2}
G.~{Lehrer} and R.~{Zhang}.
\newblock {The second fundamental theorem of invariant theory for the
  orthosymplectic supergroup}.
\newblock {\em ArXiv e-prints}, July 2014.

\bibitem[LZ2]{Lehrer-Zhang-1}
G.~Lehrer and R.~Zhang.
\newblock The first fundamental theorem of invariant theory for the
  orthosymplectic supergroup.
\newblock {\em Communications in Mathematical Physics}, pages 1--42, 2016.

\bibitem[Mar]{Mar}
P.~Martin.
\newblock The decomposition matrices of the {B}rauer algebra over the complex
  field.
\newblock {\em Trans. Amer. Math. Soc.}, 367:1797--1825, 2015.

\bibitem[{Mas}1]{Masuoka-fundamental}
A.~{Masuoka}.
\newblock {The fundamental correspondences in super affine groups and super
  formal groups.}
\newblock {\em {J. Pure Appl. Algebra}}, 202(1-3):284--312, 2005.

\bibitem[{Mas}2]{Masuoka}
A.~{Masuoka}.
\newblock {Harish-Chandra pairs for algebraic affine supergroup schemes over an
  arbitrary field.}
\newblock {\em {Transform. Groups}}, 17(4):1085--1121, 2012.

\bibitem[{Mus}]{Musson}
I.~{Musson}.
\newblock {\em {Lie superalgebras and enveloping algebras.}}
\newblock Providence, RI: American Mathematical Society (AMS), 2012.

\bibitem[Rui]{Rui}
H.~Rui.
\newblock {A criterion on the semisimple Brauer algebras}.
\newblock {\em Journal of Combinatorial Theory, Series A}, 111(1):78 -- 88,
  2005.

\bibitem[{Ser}1]{Serganova-quasireductive}
V.~{Serganova}.
\newblock {Quasireductive supergroups.}
\newblock In {\em {New developments in Lie theory and its applications.
  Proceedings of the seventh workshop, C\'ordoba, Argentina, November
  27--December 1, 2009}}, pages 141--159. Providence, RI: American Mathematical
  Society (AMS), 2011.

\bibitem[{Ser}2]{Serganova14}
V.~{Serganova}.
\newblock Finite dimensional representations of algebraic supergroups.
\newblock In {\em International Congress of Mathematicians}, 2014.

\bibitem[Ser3]{Serg}
A.~Sergeev.
\newblock Representations of the {L}ie superalgebras {$\mathfrak{ gl}(n,\,m)$}
  and {$Q(n)$} in a space of tensors.
\newblock {\em Funktsional. Anal. i Prilozhen.}, 18(1):80--81, 1984.

\bibitem[{Wat}]{Waterhouse}
W.~{Waterhouse}.
\newblock {Introduction to affine group schemes.}
\newblock {Graduate Texts in Mathematics. 66. New York, Heidelberg, Berlin:
  Springer-Verlag. XI}, 1979.

\bibitem[{Wei}]{Weissauer-semisimple}
R.~{Weissauer}.
\newblock {Semisimple algebraic tensor categories}.
\newblock {\em ArXiv e-prints}, September 2009.

\bibitem[Wen]{Wenzl}
H.~Wenzl.
\newblock On the structure of {B}rauer's centralizer algebras.
\newblock {\em Annals of Mathematics}, 128:173--193, 1988.

\bibitem[Wes]{Westra}
D.~B. Westra.
\newblock Superrings and supergroups, 2009.
\newblock available at
  \url{https://www.mat.univie.ac.at/~michor/westra_diss.pdf}.

\end{thebibliography}

\end{document}